%% file: Basak-Sarkarv4.tex
\documentclass[11pt]{amsart}
\usepackage{amsmath}
\usepackage{amssymb,amscd}
\usepackage{graphicx}
\usepackage[colorlinks=true, urlcolor=blue,bookmarks=true,bookmarksopen=true, citecolor=blue,hypertex]{hyperref}
\theoremstyle{plain}
\newtheorem{theorem}{Theorem}[section]

\newtheorem{lemma}[theorem]{Lemma}
\newtheorem{prop}[theorem]{Proposition}

\theoremstyle{definition}
\newtheorem{defn}[theorem]{Definition}

\newtheorem{example}[theorem]{Example}
\newtheorem{ques}{Question}[section]
\newtheorem{remark}[theorem]{Remark}

\def \begineq{\begin{equation}}
\def \endeq{\end{equation}}

\def \bb{\mathbb}

\def \CP{{\bb{CP}}}

\def \RR{{\bb{R}}}

\def \ZZ{{\bb{Z}}}

\def \({\left(}
\def \){\right)}
\def \<{\langle}
\def \>{\rangle}

\begin{document}

\title[Equilibrium and equivariant triangulations of small covers]
{Equilibrium and equivariant triangulations of some small covers with minimum number of vertices}

\author[B. Basak]{Biplab Basak}

\address{Department of Mathematics, Indian Institute of Science, Bangalore-560012, India.}

\email{biplab10@math.iisc.ernet.in}

\author[S. Sarkar]{Soumen Sarkar}

\address{Department of Mathematics and Statistics, University of Regina, 3737 Wascana Parkway, Regina S4S 0A2, Canada.}

\email{soumensarkar20@gmail.com}

\date{July 22, 2014}



\abstract Small covers were introduced by Davis and Januszkiewicz in 1991. We introduce the notion
of equilibrium triangulations for small covers. We study equilibrium and vertex minimal $\ZZ_2^2$-equivariant
triangulations of $2$-dimensional small covers. We discuss vertex minimal equilibrium triangulations of
$\mathbb{RP}^3 \# \mathbb{RP}^3$, $S^1 \times \mathbb{RP}^2$ and a nontrivial $S^1$ bundle over $\mathbb{RP}^2$.
We construct some nice equilibrium triangulations of the real projective space $\mathbb{RP}^n$ with
$2^n +n+1$ vertices. The main tool is the theory of small covers.
\endabstract

\maketitle

\section{Introduction}
Small covers were introduced by Davis and Januszkiewicz in \cite{[DJ]}, where an $n$-dimensional
small cover is a smooth closed manifold $N^n$ with a locally standard $\ZZ_2^n$-action such
that the orbit space of this action is a simple convex polytope. This gives a nice connection
between topology and combinatorics. A broad class of examples of small covers are all real
projective varieties. 
Illman's results \cite{[Il]} on equivariant triangulations on smooth $G$-manifolds for a finite
group $G$ ensure the existence of $\ZZ_2^n$-equivariant triangulations of $n$-dimensional small
covers. Then the following natural question can be asked. Does each $\ZZ_2^n$-equivariant
triangulation of $n$-dimensional small cover with the orbit space $Q$, induce a triangulation of $Q$?
We give the answer of this question when $n=2$, see Theorem \ref{tri2sm1}. 
It seems that no explicit equivariant triangulations are studied for small covers. On the
other hand the term ``equilibrium triangulation'' first appeared in \cite{[BK]}. They discuss
equilibrium triangulations and simplicial tight embeddings for  $\mathbb{RP}^2$ and $\CP^2$.
Inspired by the work of \cite{[BK]} and \cite{[Il]}, in this article we construct equilibrium and
equivariant triangulations of small covers explicitly with few vertices. The main results
of this article are Theorems \ref{tri2sm1}, \ref{tri2sm2}, \ref{tri3sm}, \ref{rp33},
\ref{nn32}, \ref{s1rp2} and \ref{tripro}. Main tool in this article is the theory of small covers.

The arrangement of this article is as follows. We recall the definitions and some basic
properties of small covers following \cite{[DJ]} in Subsection \ref{smcov}. In Subsection
\ref{pritri}, we review the definitions and some results on triangulations of manifolds
following the survey \cite{[Da]}. The definition of equilibrium triangulation of small cover
is introduced in the subsection \ref{eqitri} following an idea of \cite{[BK]}. Some
$\ZZ_2^n$-equivariant triangulations of an $n$-dimensional small cover are constructed
in Subsection \ref{trismm}. These simple looking triangulations lead to some open problems,
see Subsection \ref{trismm}. We discuss some vertex minimal $\ZZ_2^2$-equivariant triangulations
of $2$-dimensional small covers in Subsection \ref{sm2d}. Subsection \ref{sm3d} gives some natural
equilibrium triangulations of $3$-dimensional small covers. In this subsection, we study
vertex minimal equilibrium triangulations of $\mathbb{RP}^3 \# \mathbb{RP}^3$, a nontrivial
$S^1$ bundle over $\mathbb{RP}^2$ and $S^1 \times \mathbb{RP}^2$, see Theorem \ref{rp33},
\ref{nn32} and \ref{s1rp2} respectively. Subsection \ref{triproj} gives some explicit nice
equilibrium triangulations of real projective space $\mathbb{RP}^n$ with $2^n +n+1$ vertices.

\section{Preliminaries}\label{pre}
\subsection{Small covers}\label{smcov}
Following \cite{[DJ]}, we discuss some basic results about small covers. The codimension one
faces of a convex polytope are called $facets$. An $n$-dimensional $simple$ polytope in $\RR^n$ is a
convex polytope with each vertex is the intersection of exactly $n$ facets. We denote the underlying
additive group of the vector space $\mathbb{F}_2^n$ by $\ZZ_2^n$. Let $\rho : \ZZ_2^n \times \RR^n
\to \RR^n$ be the standard action. Let $N$ be an $n$-dimensional manifold.
\begin{defn}
An action $\eta :\ZZ_2^n \times N \to N$ is said to be {\em locally standard} if the following holds.
$(1)$  Every point $y \in N $ has a $\ZZ_2^n$-stable open neighborhood $U_y$.
$(2)$ There exists a homeomorphism $\psi : U_y \to V$, where $V$ is a $\ZZ_2^n$-stable open subset of $\RR^n$. 
$(3)$ There exists an isomorphism $\delta_y : \ZZ_2^n \to \ZZ_2^n$ such that
$\psi( \eta (t, x)) = \rho(\delta_y (t), \psi(x))$ for all $(t,x) \in \ZZ_2^n \times U_y$.
\end{defn}
\begin{defn}
A closed $n$-dimensional manifold $N$ is said to be a {\em small cover} if there is an effective $\mathbb{Z}_2^n$-action
 on $N$ such that: $(1)$ the action is a locally standard action,
$(2)$ the orbit space of the action is a simple polytope (possibly diffeomorphic as manifold with corners to a simple polytope).
\end{defn}
Let $N$ be a small cover and $\xi: N \to Q$ be the orbit map. We say that $N$ is a small cover over
$Q$ or $\xi: N \to Q$ is a small cover over $Q$.
Let $Q^0$ be the interior of $Q$. We call the subset $\xi^{-1}(Q^0)$ the {\em principal fibration} of $N$.

\begin{example}
The natural action of $\ZZ_2^n$ defined on the real projective space $\mathbb{RP}^n$ by
\begin{equation}
(g_1, \ldots, g_n)\cdot [x_0, x_1, \ldots, x_n] \to [x_0, (-1)^{g_1}x_1, \ldots, (-1)^{g_n}x_n]
\end{equation}
is locally standard and the orbit space is diffeomorphic as manifold with corners to the
standard $n$-simplex. Hence $\mathbb{RP}^n$ is a small cover over the $n$-simplex $\bigtriangleup^n$.
\end{example}

\begin{remark}
The equivariant connected sum of $n$-dimensional finitely many small covers
is also a small cover of same dimension. Details can be found in page-424 of \cite{[DJ]}.
\end{remark}

\begin{defn}
Let $\mathcal{F}(Q) = \{F_1, \ldots, F_{m}\}$ be the set of facets of a simple $n$-polytope $Q$.
A function $\beta : \mathcal{F}(Q) \to \mathbb{F}_2^n$ is called a $\ZZ_2$-{\em characteristic function}
on $Q$ if the span of $\{\beta(F_{j_1}), \ldots, \beta(F_{j_l})\}$ is an $l$-dimensional subspace of
$\mathbb{F}_2^n$ whenever the intersection of the facets $F_{j_1}, \ldots, F_{j_l}$ is nonempty.
The vector $ \beta_j := \beta(F_{j})$ is called $\ZZ_2$-characteristic vector assigned to $F_j$ and the pair
$(Q, \beta)$ is called $\ZZ_2$-characteristic pair.
\end{defn}

Note that given a small cover, we can define a $\ZZ_2$-characteristic pair, see Section 1 in \cite{[DJ]}.
Let $(Q, \beta)$ be a $\ZZ_2$-characteristic pair where $Q$ is a simple $n$-polytope with facets $F_1,
\ldots, F_m$. Let $G_F$ be the subgroup of $\ZZ_2^n$ generated by  $\{\beta_{j_1}, \ldots,$ $\beta_{j_l}\}$,
whenever $F = F_{j_1} \cap \ldots \cap F_{j_l}$. Define an equivalence relation $\sim$ on $\mathbb{Z}_2^n \times Q $ by
\begin{equation}\label{equi}
(t,p) \sim (s,q) ~ \mbox{if} ~ p = q ~ \mbox{and} ~ s-t \in G_{F},
\end{equation}
where $F \subset Q$ is the unique face whose relative interior contains $p$. We denote
the equivalence class of $(t,p)$ by $[(t, p)]^{\sim}$. The quotient space
$ N(Q,\beta) = (\mathbb{Z}_2^n \times Q)/\sim$ is an $n$-dimensional small cover, see Section 1 in \cite{[DJ]}
for more details. Following Proposition gives a classification of small covers.

\begin{prop}[Proposition 1.8, \cite{[DJ]}]\label{clasmc}
Let $\xi : N \to Q $ be a small cover over $Q$ and the function $\beta : \mathcal{F}(Q) \to \mathbb{F}_2^n$
be its $\ZZ_2$-characteristic function. Let $\xi_{\beta} : N(Q, \beta) \to Q$
be the constructed small cover from the $\mathbb{Z}_2$-characteristic pair $(Q, \beta)$.
Then there exists an equivariant homeomorphism from $N$ to $N(Q, \beta)$ covering the identity over $Q$.
Hence a small cover $\xi: N \to Q$ is determined up to equivalence over $Q$ by its $\ZZ_2$-characteristic function.
\end{prop}
The $cubical~ subdivision$ of a simple polytope is explicitly discussed in Section $4.2$ of \cite{[BP]}, a
good reference for many interesting developments and applications. We denote a cubical subdivision of the
polytope $Q$ by $C(Q)$. We consider the vertex in $C(Q)$ corresponding to $Q$ is the center of mass of $Q$.

\subsection{Triangulation of manifolds}\label{pritri}
We recall some basic definitions for triangulations of manifolds following \cite{[RS]}.
A compact convex polyhedron which spans a subspace of dimension $n$ is called an $n$-cell. So
we can define faces of an $n$-cell.
\begin{defn}
A $cell ~ complex$ $X$ is a finite collection of cells in some $\RR^n$ satisfying,
(i) if $B$ is a face of $A$ and $A \in X$ then $B \in X$, (ii) if $A, B \in X$ and $A \cap B$ is nonempty
then $A \cap B$ is a face of both $A$ and $B$. Zero dimensional cells of $X$ are called vertices of $X$.
A cell complex $X$ is $simplicial$ if each $A \in X$ is a simplex.
\end{defn}

We may denote a simplex $\sigma$ with vertices $v_1, \ldots, v_k$ by $v_1v_2\ldots v_k$.
The vertex set of a simplicial complex $X$ is denoted by $V(X)$.
The union of all simplices in a simplicial complex $X$ is called the {\em geometric carrier} of $X$ which
is denoted by $|X|$.
\begin{defn}
If a Hausdorff topological space $M$ is homeomorphic to $|X|$, the geometric carrier of a simplicial
complex, then we say that $X$ is a {\em triangulation} of $M$. If $|X|$ is a topological $d$-ball
(respectively, $d$-sphere) then $X$ is called a {\em triangulated $d$-ball} (resp., {\em triangulated
$d$-sphere}). 
\end{defn}
\begin{defn}
If $X$, $Y$ are two simplicial complexes, then a {\em simplicial
map} from $X$ to $Y$ is a continuous map $\eta : |X| \to |Y|$
such that for each $\sigma \in X$, $\eta(\sigma)$ is a simplex of $Y$
and $\eta|_{\sigma}$ is linear.
\end{defn}

By a {\em simplicial subdivision} of a cell complex  $X$ we mean a
simplicial complex $X^{\hspace{.1mm}\prime}$ together with a
homeomorphism from $|X^{\hspace{.1mm}\prime}|$ onto $|X|$ which is face-wise linear.
From the following proposition we can construct a simplicial complex from a cell complex.
\begin{prop}[2.9 Proposition, \cite{[RS]}]\label{thmrs}
 A cell complex can be subdivided to a simplicial complex without introducing any new vertices.
\end{prop}
\begin{defn}
 Let $G$ be a finite group. A $G$-$equivariant~ triangulation$ of the $G$-space $N$ is a triangulation $X$ of $N$ such
that if $\sigma \in X$ then $f_g(\sigma) \in X$ for all $ g \in G$, where $f_g$ is the
homeomorphism corresponding to the action of $g$ on $|X|$.
\end{defn}

\section{Equilibrium and equivariant triangulations of small covers}
\subsection{Definition of equilibrium Triangulation}\label{eqitri}
We generalize the definitions of equilibrium set and zones of influence
for small covers following the definition of equilibrium triangulation
for $\mathbb{RP}^2$ and $\CP^2$ of \cite{[BK]}. These definitions are also
generalized for quasitoric manifolds in \cite{[DS]}. Let $\xi: N \to Q$
be an $n$-dimensional small cover over the $n$-dimensional simple polytope $Q$.
\begin{defn}
The $equilibrium~ set$ of an $n$-dimensional small cover $N$ is defined
by the orbit at $\tilde{x} \in N$ such that $\xi(\tilde{x})$
is the center of mass of $Q$.
\end{defn}
\begin{defn}
A $zones ~of ~influence$ of an $n$-dimensional small cover $N$ is defined by a
$\ZZ_2^n$ invariant closed subset such that
$(1)$ it contains the equilibrium set,
$(2)$ it is $\ZZ_2^n$-equivariantly homeomorphic to a $\ZZ_2^n$-invariant closed
ball $B^{n} \subset \RR^n$.
\end{defn}
\begin{defn}
A collection $\{Z_1, \ldots, Z_k\}$ of zones of influences of an $n$-dimensional
small cover $N$ is called $complete$ if the collection $\{\xi(Z_{i_1}) \cap 
\ldots \cap \xi(Z_{i_l}): \{i_1, \ldots, i_l\} \subseteq \{1, \ldots, k\} \}$
gives a cubical subdivision of $Q$.
\end{defn}
Let $m$ be the number of vertices in $Q$. If $C(Q)$ is a cubical subdivision
of $Q$ with $n$-dimensional cell $I_i$ for $1 \leq i \leq m$ then the collection $\{\xi^{-1}(I_i): i=1,
\ldots, m\}$ of zones of influences is complete.
\begin{defn}
A triangulation of $n$-dimensional small cover $N$ is said to be an $equilibrium$ 
$triangulation$ if the equilibrium set and all the zones of influences of a complete collection  are  triangulated
submanifolds of $N$.
\end{defn}

\begin{remark}
 In \cite{[BK]}, equilibrium triangulation is used for $\mathbb{RP}^2$ and $\CP^2$. The equilibrium 
triangulation on $\mathbb{RP}^2$ induces a subdivision of the orbit space, namely triangle, in equilibrium
$2$-cells. Now consider a regular $n$-simplex and a regular $n$-cube.
In the cubical subdivision of these polytopes, the same dimensional cells are in equilibrium position
around the center. Observing this phenomena, we use the name equilibrium triangulation for small cover
over any polytope. For example, let $Q$ be a regular $n$-simplex or a regular $n$-cube with vertices 
$v_1, \ldots, v_k$. Then the equilibrium set is $\xi^{-1}(C_Q)$, where $C_Q$ is the center of $Q$. The set of zones of
influences is $\{\xi^{-1}(I_i) : I_i ~ \mbox{is the}\, n\mbox{-cell in cubical subdivision of}~ Q,~\mbox{corresponding to the vertices}$
$v_i, \, 1\leq i \leq k\}$. 
\end{remark}

Any equilibrium triangulation of an $n$-dimensional small cover $N$ contains at least
$2^n$ vertices. We study some equilibrium and equivariant triangulations of small covers
in the following subsection.

\subsection{Some triangulations of small covers}\label{trismm}
Let $\xi : N \to Q $ be a small cover over the $n$-dimensional simple polytope $Q$. Let $Q^{\prime}$
be the first barycentric subdivision of $Q$. Since $Q$ is a polytope, $Q^{\prime}$ is a triangulation
of $Q$. Let $\alpha$ be an $l$-dimensional simplex in $Q^{\prime}$ and $\alpha^0$ be the relative interior of
$\alpha$. From the equivalence relation $\sim$ in the Equation (\ref{equi}), it is clear that $\xi^{-1}(\alpha^0)$ is
a disjoint union of $\sigma_{1}(\alpha)^0, \cdots, \sigma_{2^k}(\alpha)^0$ in $N$ where $\alpha$ belongs to the smallest
face $F^k$ of $Q$ of dimension $k$ and $\sigma_i(\alpha)$ is the closure of $\sigma_i(\alpha)^0$ in $N$.
The restriction of $\xi$ on each $\sigma_{i}(\alpha)^0$ is a diffeomorphism to $\alpha^0$. Note that small covers have
smooth structure (cf. \cite{[DJ]}). The subset $\sigma_{i}(\alpha)$ of $N$ is diffeomorphic as manifold with corners to
an $l$-dimensional simplex. Clearly the collection $$\Sigma(N)=\{\sigma_{i}(\alpha) : \alpha \in Q^{\prime}, \alpha
\subset F^k ~\mbox{and} ~i= 1, \ldots, 2^k\}$$ gives
a triangulation of $N$. Here $k$ also depends on $\alpha$. Note that $\Sigma(N)$ is an equivariant
triangulation of $N$ with respect to the $\ZZ_2^n$ action. Also this is an equilibrium triangulation
of $N$. With the triangulations $Q^{\prime}$ and $\Sigma(N)$ of $Q$ and $N$ respectively the projection
map $\xi$ is simplicial. Hence we get the following.

\begin{lemma} \label{lemma:triangulation}
A small cover $N$ over $Q$ has a $\ZZ_2^n$-equivariant and equilibrium triangulation with $2^n+ f_0 2^{n-1}
+ f_12^{n-2}+ \cdots +f_{n-1}$ vertices, where $f(Q)=(f_0, f_1, \ldots, f_{n-1})$ is the face vector of $Q$.
Moreover, Euler characteristic of all small covers over $Q$ are same. 
\end{lemma}

\begin{remark}
Let $K$ be a contracted pseudotriangulation (crystallization) of an $n$-manifold $M$ with vertex set $\{v_1,
\dots, v_{n+1}\}$ and with $f_i(K)$ $i$-simplices for $1 \leq i \leq n$ (for more on crystallizations and contracted
pseudotriangulations of manifolds see \cite{[BD], [FGG]}). Let $\sigma^n$ be a simplex with vertex set $\{1,
\dots, n+1\}$. For any set $\{v_{i_0}, \dots, v_{i_j}\}$ of vertices, number of $j$-faces in $K$ whose vertices
are $v_{i_0}, \dots, v_{i_j}$ is in general more than one and the map $\varphi$ given by $v_i \mapsto i$ is not a
piece wise linear (pl) map. To make it a pl we have to first subdivide $K$ to a simplicial complex. Clearly, we
have to add at least $\Sigma_{i=1}^{n}(f_i(K) - \binom{n+1}{i+1})$ new vertices to make the subdivision a
simplicial complex (if $\alpha_1, \dots, \alpha_k$ are $i$-simplices with same vertex set then we add one vertex
in each $\alpha_2, \dots, \alpha_k$). To make the map $\varphi$ simplicial, we need to add
$\Sigma_{i=1}^{n}f_i(K)$ vertices. So, the new simplicial complex $\tilde{K}$ (which is a sub division of $K$)
has $(n+1)+\Sigma_{i=1}^{n}f_i(K)$ vertices (and $2^{n+1}+2$ vertices in the subdivision $\widetilde{\sigma}$ of
$\sigma^n$). This simplicial map $\varphi : \widetilde{K} \to \widetilde{\sigma}$ is also a branched covering
with quotient polytope a simplex. We have considered such cases in Lemma \ref{lemma:triangulation} where the
quotient polytopes are in general simple polytopes. Our main results (cf. Theorems \ref{rp33}, \ref{nn32}, \ref{s1rp2}, \ref{tripro})
show that we can have branched coverings over simple polytopes with much less number of additional vertices on the manifolds, namely,
we add only a vertex in the interior of each $n$-simplex.
\end{remark}

Observing the above natural triangulations of small covers we may ask the following.

\begin{ques}\label{ques1}
Let $\xi : N \to Q $ be a small cover over the $n$-dimensional simple polytope $Q$. What are the
vertex minimal equilibrium triangulations of $N$?
\end{ques}
\begin{ques}\label{ques2}
Let $\xi : N \to Q $ be a small cover over the $n$-dimensional simple polytope $Q$. Can we describe
all $\ZZ_2^n$-equivariant triangulations of $N$ with minimum vertices?
\end{ques}

\subsection{Triangulations of $2$-dimensional small covers}\label{sm2d}
In this subsection, we construct some equilibrium and equivariant triangulations of $2$-dimensional small
covers with few vertices. We calculate the number of vertices and hence faces of these triangulations.
Let $Q$ be a polygon with vertices $V_1, \ldots, V_{m}$ and edges $F_1, \ldots,
F_{m}$, where the end points of $F_i$ are $V_i$ and $V_{i+1}$ for $i= 1, \ldots, m-1$
and the end points of $F_m$ are $V_m$ and $V_1$. Let $\xi: N \to Q$ be a small cover over $Q$.
We denote the fixed point corresponding to the vertex $V_i$ by the same.  Let $\beta: \{F_1, \ldots, F_m\} \to \ZZ_2^2$
be the $\ZZ_2$-characteristic function of $N$. Let $\beta_{i_1}= \cdots= \beta_{i_{k_1}}=(1,0)$,
$\beta_{j_1}= \cdots= \beta_{j_{k_2}}=(0,1)$ and $\beta_{l_1}= \cdots= \beta_{l_{k_3}}=(1,1)$.
So $k_1 + k_2+ k_3 = m$, $k_i \geq 0$ and at most one $k_i$ is zero.
\begin{lemma}\label{eqtrism1}
There is a $\ZZ_2^2$-equivariant and equilibrium triangulation of $2$-dimensional small cover
$N$ (over $m$-gon $Q$) with $2m+4$ vertices. Euler characteristic of $N$ is $4-m$.
\end{lemma}
\begin{proof}
Let $C(Q)$ be the cubical subdivision of $Q$. Let $C_F$ be the vertex in $C(Q)$ corresponding to
the nonempty face $F$ of $Q$. Joining the vertices $C_{F_{i-1}}$ and $C_{F_i}$ in $C(Q)$ for $i=1,
\ldots, m$ with $F_0=F_m$, we get a triangulation $Y$ of $Q$. Note that $N$ is equivariantly diffeomorphic to
$N(Q, \beta) := (\mathbb{Z}_2^2 \times Q)/\sim$ by Proposition \ref{clasmc}, where $\sim$ is
defined in the Equation (\ref{equi}). Let $\xi^{-1}(C_{Q}) =\{C_{\beta} \in N : \beta \in \ZZ_2^2\}$ and
$\xi^{-1}(C_{F_i})=\{C_i^{\prime}, C_i^{\prime \prime}\}$. Then $\xi^{-1}(V_iC_{F_i}C_{F_{i-1}})$
is the square with edges $[C_i^{\prime}, C_{i-1}^{\prime}], [C_{i-1}^{\prime}, C_i^{\prime \prime}],
[C_i^{\prime \prime}, C_{i-1}^{\prime \prime}], [C_{i-1}^{\prime \prime}, C_i^{\prime}]$ for $i=1, \ldots, m$. We consider the collection
\begin{equation}
\begin{array}{cccc}
\Sigma(N)=\{[C_i^{\prime}, C_{i-1}^{\prime}, C_{i-1}^{\prime \prime}], [C_i^{\prime \prime},
C_{i-1}^{\prime}, C_{i-1}^{\prime \prime}] : i= 1, \ldots, m\}\\ \cup \{[C_{(0, 0)}, C_i^{\prime},
C_{i-1}^{\prime}], [C_{\beta_i}, C_i^{\prime}, C_{i-1}^{\prime \prime}], [C_{\beta_{i-1}},
C_i^{\prime \prime}, C_{i-1}^{\prime}], [C_{\beta_i + \beta_{i-1}}, C_i^{\prime \prime}, C_{i-1}^{\prime \prime}] : i=1, \ldots, m\}
\end{array}
\end{equation}
Observe that this gives a $\ZZ_2^2$-equivariant
triangulation of $N$. Clearly, this is also an equilibrium triangulation of $N$. The face vector of this
triangulation is $(2m+4, 9m, 6m)$. So Euler characteristic of $N$ is $4-m$.
\end{proof}

\begin{lemma}
If $k_i=1$ for some $i$ then there is a $\ZZ_2^2$-equivariant triangulation of $N$ with $2m$ vertices.
\end{lemma}
\begin{proof}
We may assume $k_1= 1$. Let $\xi : N \to Q$ be the orbit map, where $Q$ is an $m$-gon with
vertices $V_1, \ldots, V_m$ and edges $F_1, \ldots, F_m$. Let $C_i$ be the middle
point of the edge $F_i$ for $i=1, \ldots, m$. We consider the collection $$Y = \{[V_i, C_i, C_{i-1}] : i = 1, \ldots, m\} 
\cup \{[C_1, C_i, C_{i+1}] : i = 2, \ldots, m-1\}$$ of $2$-simplices in $Q$ with $C_0 = C_m$. Clearly it gives a triangulation
of $Q$. Since $k_1=1$, with this triangulation of $Q$, similarly as in Lemma \ref{eqtrism1}, we
can construct a $\ZZ_2^2$-equivariant triangulation of $N$ which does not contain any fixed points.
Note that this is not an equilibrium triangulation of $N$.
\end{proof}

\begin{example}\label{trid2}
We discuss some $\ZZ_2^2$-equivariant triangulations of $2$-disc $D^2=\{(x_1, x_2) \in \RR^2:
x_1^2 + x_2^2 \leq 1\}$. Let $g_1, g_2$ be the standard generator of $\ZZ_2^2$ and $g_i$ acts
by reflection along $\{(x_1, x_2):x_i=0\}$ for $i=1,2$. Following Figure \ref{egch201a} gives all possible $\ZZ_2^2$-equivariant triangulations
of $D^2$ with $4$ and $5$ vertices. Clearly, any $\ZZ_2^2$-equivariant triangulation not containing
the fixed point $(0,0) \in D^2$ contains exactly $4$ vertices and the triangulations are given by
Figure \ref{egch201a} $(1)$ and $(2)$. If a $\ZZ_2^2$-equivariant triangulation of $D^2$ contains
more vertices, the corresponding triangulation can be obtained by adding some edges joining antipodal points
of boundary of $\partial{D}^2$.
\begin{figure}[ht]
\centerline{
\scalebox{.74}{
\input{egch201a.pstex_t}
 }
 }
\caption {$\ZZ_2^2$-equivariant triangulations of $D^2$ with $4$ and $5$ vertices.}
\label{egch201a}
\end{figure}
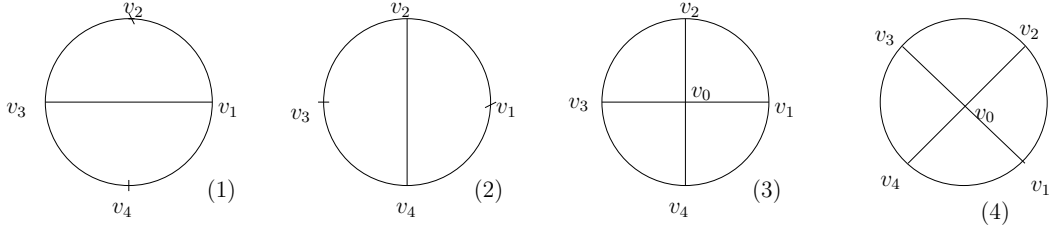
\end{example}

\begin{theorem}\label{tri2sm1}
Let $X$ be a $\ZZ_2^2$-equivariant triangulation of the small cover $\xi: N \to Q$ over the
polygon $Q$. Then $A = \{\xi(\sigma): \sigma \in X ~ \mbox{is a}~2\mbox{-simplex}\}$
is the set of $2$-simplices in a triangulation of $Q$.
\end{theorem}
\begin{proof}
First we show that if $\sigma$ is a $2$-simplex in $X$ then $\xi(\sigma)$ is a triangle (possibly
diffeomorphic to a triangle). This is true if $\sigma \cap \xi^{-1}(F_i)$ is empty for $i=1,
\ldots, m$. From the discussion in example \ref{trid2}, it is clear that if $\sigma$ contains
a fixed point then $\xi(\sigma)$ is a triangle.

Let the intersection $\sigma \cap \{\xi^{-1}(F_{j}) \backslash \{V_{1}, \ldots, V_{m}\}\}$ be nonempty for some
$j \in \{1, \ldots, k\}$. Observe that $\xi^{-1}(F_{j}) \backslash \{V_{j}, V_{j +1}\}$ is disjoint
union of two open arcs, namely $C_{1}$ and $C_{2}$, in $N$. Then either $\sigma \cap C_{1}$ or
$\sigma \cap C_{2}$ is empty. Suppose both  $\sigma \cap C_{1}$ and $\sigma \cap C_{2}$ are nonempty.
Then there is an edge $e$ of $\sigma$ such that $e \cap C_{1}$ and $e \cap C_{2}$ are nonempty, see Figure
\ref{egsm1}. Since $\sigma$ does not contain any fixed point and $X$ is $\ZZ_2^2$-equivariant, $\beta_j(e) \neq e$
and $\beta_j(e)$ is an edge in $X$, where $\beta_{j}$ is the $\ZZ_2$-characteristic vector associated to
the facet $F_{j}=[V_i, V_{j+1}]$. Clearly $e \cap \beta_j(e)$ contains at least 2 points (see Figure \ref{egsm1}),
which is a contradiction.
Now assume $\sigma \cap C_{1}$ is nonempty. So either $\sigma \cap C_{1}$ is a point or a closed interval.
\begin{figure}[ht]
\centerline{
\scalebox{.70}{
\input{egsm1.pstex_t}
 }
 }
\caption {}
\label{egsm1}
\end{figure}

Suppose $\sigma \cap C_{1}$ is a point. Then it is a vertex $v_1$ of $\sigma$.
Let $v_2 $ and $v_3$ be the other two vertices of $\sigma$. Since $\sigma$ does not contain any fixed
point, $\sigma - \{[v_2, v_3] \cup v_1\}$ is a subset of a connected component of the principal
fibration of $N$. So $\xi(\sigma)$ is a triangle in $Q$.

Suppose $\sigma \cap C_{1}$ is a closed interval $[a, b]$. 
Since $X$ is $\ZZ_2^2$-equivariant triangulation, at least one of $a, b$
belongs to $V(\sigma)$. Let both $a, b$ belong to $V(\sigma)$ and $c$ be the third vertex of $\sigma$.
Then by equivariance condition on $X$ we get $[a, b] \subset C_{1}$ and $\sigma - \{c\}$ belongs
to a connected component of the principal fibration. So $\xi(\sigma)$ is a triangle in $A$.

Let $a$ belongs to $V(\sigma)$ and $c, d$ be the other two vertices of $\sigma$. Then by equivariance
condition on $X$ we get $\xi(c)= d$ and $\beta_{j}([a,c,d])=[a, c, d]$, where $\beta_{j}$ is
the $\ZZ_2$-characteristic vector associated to the facet $F_{j}$. Then $\xi([a,c,d])= [a,b,c]$.
By previous arguments on $[a,b,c]$, $\xi(\sigma)$ is a triangle in $Q$. We draw some pictures for
the above situations respectively in Figure \ref{egch201b}.

Let $\tau_1$ and $\tau_2$ be $2$-simplices in $A$. So $\xi(\sigma_1) = \tau_1$ and $\xi(\sigma_2) = \tau_2$
for some $2$-simplices $\sigma_1$ and $\sigma_2$ in $X$. Then $\sigma_2 \neq g \sigma_1$ for all $g \in \ZZ_2^2$.
Since the action is locally standard, we have $\tau_1 \cap \tau_2 = \xi(\sigma_1) \cap \xi(\sigma_2)
= \xi(\sigma_1 \cap \sigma_2)$ and $\xi(\sigma_1 \cap \sigma_2)$ is diffeomorphic to a closed interval or point or empty.
Hence we get the theorem.
\begin{figure}[ht]
\centerline{
\scalebox{.70}{
\input{egch201b.pstex_t}
 }
 }
\caption {}
\label{egch201b}
\end{figure}
\end{proof}

\begin{theorem}\label{tri2sm2}
Any $\ZZ_2^2$-equivariant triangulation of a $2$-dimensional small cover $N$
over $m$-gon $Q$ contains at least $2m$ vertices.
\end{theorem}
\begin{proof}
Let $X$ be a $\ZZ_2^2$-equivariant triangulation of the small cover $N$ over $m$-gon $Q$.
Let $V_1, \ldots, V_m$ be the vertices of $Q$. Then $\{F_i :=[V_i, V_{i+1}] : i=1, \ldots, m-1\} \cup \{F_m := [V_1, V_m]\}$
is the set of facets  of $Q$. Let $[a,b,c]$ be a $2$-simplex in $X$ such that $V_i \in [a,b,c]$ for some $i \in \{1,
\ldots, m\}$. From the discussion in Example \ref{trid2}, either $V_i$ is a vertex or $V_i$ is an interior
point of an edge of $[a,b,c]$. Our claim is that $$[a,b,c] \cap \{\cup_{V_i \notin F_j} \xi^{-1}(F_j)\}$$ is empty.
Let $[a,b,c] \cap \xi^{-1}(F_j)$ be nonempty with $V_i \notin F_j$ for some $j$.
Then either $[a, b] \cap \xi^{-1}(F_j) ~\mbox{or} ~[a, c] \cap \xi^{-1}(F_j) ~\mbox{or}
~ [b, c] \cap \xi^{-1}(F_j)$ is nonempty. 

First assume that $V_i=a$. Suppose $d \in [a, b] \cap \xi^{-1}(F_j)$ and $\beta_j$ is the $\ZZ_2$-characteristic
function assign to $F_j$. Since $X$ is equivariant, $\beta_j([a, b])$ is an edge in $X$ and $\beta_j([a, b]) \neq [a, b]$.
Then $\{a, d\} \subset [a, d] \cap \xi^{-1}(F_j)$, which is a contradiction.

Now assume $V_i$ belongs to the interior of $[a,b]$. Then $[a,b,c]$ must be similar to one of the
$2$-simplices in the Figure \ref{egch201a} $(1)$ and $(2)$. So $\{a,b,c\} \subset \xi^{-1}(F_i)
\cup \xi^{-1}(F_{i-1})$. This proves the claim.

 Now we assign two distinct vertices $x_{i_1}, x_{i_2} \in V(X)$ to the edge $F_i$ for each
$i \in \{1, \ldots, m\}$ such that all $x_{i_j}$'s are distinct. We do it step by step in
following way. Since $X$ is equivariant, $\{\xi(\sigma): \sigma \in X ~ \mbox{is a}~2\mbox{-simplex}\}$
is the set of $2$-simplices in a triangulation $Y$ of $Q$ (by Theorem \ref{tri2sm1}).

{\bf Step $1$}: First consider the edge $F_1$. Let $\tau_1$ be a $2$-simplex in $Y$ such that $V_1 \in \tau_1$
and $\tau_1 \cap F_1= [V_1, y_1]$ is an edge of $\tau_1$. Let $\xi(\sigma_1) =\tau_1$ for some $2$-simplex
$\sigma_1$ in $X$. So by Example \ref{trid2}, either $V_1$ is a vertex of $\sigma_1$ or $V_1$ belongs to
the interior of an edge of $\sigma_1$.

Suppose $V_i$ belongs to the interior of an edge of $\sigma_1$. Then
$\sigma_1$ must be similar to one of the $2$-simplices in the Figure \ref{egch201a} $(1)$ and $(2)$.
So $V(\sigma_1) \cap \xi^{-1}(y_1)$ is nonempty and it is a vertex of $\sigma_1$. Since $y_1$ belongs
to relative interior of the edge $F_1$ of $Q$, $\xi^{-1}(y_1) = \{x_{1_1}, x_{1_2}\} \subset \xi^{-1}(F_1^0)
\subset N$. Then we choose the vertices $\{x_{1_1}, x_{1_2}\} \subset V(X) $ for $F_1$.

Suppose $V_1$ is a vertex of $\sigma_1$ and $V(\sigma_1) \cap \xi^{-1}(y_1)$ is nonempty. So
$V(\sigma_1) \cap \xi^{-1}(y_1)$ is a vertex of $\sigma_1$. Then we choose the vertices
$\{x_{1_1}, x_{1_2}\} \subset \xi^{-1}(F_1^0) \cap V(X) $ for $F_1$.

Suppose $V_1$ is a vertex of $\sigma_1$ and $V(\sigma_1) \cap \xi^{-1}(y_1)$ is empty.
Recall that $\beta_i$ is the $\ZZ_2$-characteristic vector assigned to $F_i$. Since $X$ is
equivariant, there is an edge $[x_{1_1}, x_{1_2}]$ of $\sigma_1$ not containing $V_1$ such that
$\beta_1([x_{1_1}, x_{1_2}]) = [x_{1_1}, x_{1_2}]$ and $[x_{1_1}, x_{1_2}] \cap \xi^{-1}(y_2) =
\{x_{1_2}, x_{1_2}\}$ where $y_2 \in Q^0$ is the third vertex of $\tau_1$. 
We may assume $x_{1_1}= [((0,0), y_2)]^{\sim}$ and $\beta_1=(1,0)$.
So $x_{1_2}= [((1,0), y_2)]^{\sim}$. Let $x_{1_1}^{\prime} = [((0,1), y_2)]^{\sim}$
and $x_{1_2}^{\prime} = [((1,1), y_2)]^{\sim}$. Then the $2$-simplices $[V_1, x_{1_1},
x_{1_2}]$ and $[V_1, x_{1_1}^{\prime}, x_{1_2}^{\prime}]$ belong to $X$. Then we
assign $\{x_{1_1}, x_{1_2}\}$ to $F_1$. We explain this situation in Figure \ref{egsm2}.
\begin{figure}[ht]
\centerline{
\scalebox{.70}{
\input{egsm2.pstex_t}
 }
 }
\caption {}
\label{egsm2}
\end{figure}

{\bf Step $2$}: Since $\{\beta_1, \beta_2\}$ is a basis of $\ZZ_2^2$, we may assume $\beta_2=(0,1) \in \ZZ_2^2$.
Similarly as step $1$, we get either $\xi^{-1}(F_2^0)$ contains at least $2$ vertices,
namely $x_{2_1}$ and $x_{2_2}$, of $X$ or there are $2$-simplices $[V_2, x_{2_1}, x_{2_2}]$
and $[V_2, x_{2_1}^{\prime}, x_{2_2}^{\prime}]$ in $X$ with $x_{2_1}$ belongs to
principal fibration of $N$. We may assume $x_{2_1}= [((0,0), y_4)]^{\sim}$ for
some $y_4 \in Q^0$. If $y_2=y_4$, then $x_{1_1}= x_{2_1}$ and either $x_{1_2}=
x_{2_1}^{\prime}$ or $x_{1_2} =x_{2_2}^{\prime}$. Suppose $x_{1_2}= x_{2_1}^{\prime}$.
So we choose the vertices $x_{2_2}, x_{2_2}^{\prime}$ for $F_2$. If $y_2 \neq y_4$,
we get $4$ new vertices and we assign two from them to $F_2$.

{\bf Step $3$}: For $F_3$ we have two possibilities, either $\beta_3=(1,0)$ or $\beta_3 = (1,1)$.
In either case similarly as step $1$, we get either $\xi^{-1}(F_3^0)$ contains at least
$2$ vertices, namely $x_{3_1}$ and $x_{3_2}$, of $X$ or there are $2$-simplices $[V_3, x_{3_1},
x_{3_2}]$ and $[V_3, x_{3_1}^{\prime}, x_{3_2}^{\prime}]$ in $X$ with $x_{3_1}$ belongs
to principal fibration of $N$. We may assume $x_{3_1}= [((0,0), y_6)]^{\sim}$ for some
$y_6 \in Q^0$. Suppose $\beta_3=(1,0)$. Then $y_2 \neq y_6$, otherwise we get two $1$-simplices
with same vertices. Assume $y_4 = y_6$ then similarly as in step $2$, we can
assign two vertices different from previous steps to the edge $F_3$. If $y_4 \neq y_6$,
we get $4$ new vertices and we assign two from them to $F_3$.
Now Suppose $\beta_3=(1,1)$. If $y_6$ is different from one of $y_2, y_4$, we can proceed
similarly as step $2$ to assign two new vertices to $F_3$. Let $y_2=y_4=y_6$. Then we
choose $x_{3_1}=V_1$ and $x_{3_2}=V_2$.

Note if $\beta_1=(1,0), \beta_2=(0,1), \beta_3 =(1,1)$ and $y_2=y_4=y_6$, then there are edges
between any two points of $\xi^{-1}(y_2)$. So $y_{2k} \neq y_2$ for any $3 < k \leq m$.
Continuing in this way we can assign two vertices different form the previous vertices
to each edge $F_i$ for $i >1$. This proves the theorem. 
\end{proof}

\begin{remark}
If $X$ is a $\ZZ_2^2$-equivariant vertex minimal triangulation of a $2$-dimensional small cover $N$ with $m$
fixed points, then by Lemma \ref{eqtrism1} and Theorem \ref{tri2sm2}, $2m \leq |V(X)| \leq 2m +4$.
\end{remark}

\begin{example}\label{rp2}
A $12$-vertex, $6$-vertex and a $7$-vertex equivariant triangulations of torus and $\mathbb{RP}^2$ are
given by the following Figure \ref{egch204} respectively. The minimal triangulation of $\mathbb{RP}^2$ is given
by $6$ vertices. So this $6$-vertex triangulation is the vertex minimal $\ZZ_2^2$-equivariant triangulation
of $\mathbb{RP}^2$. By our construction it is clear that the $7$-vertex triangulation is a
vertex minimal equilibrium triangulation of $\mathbb{RP}^2$.
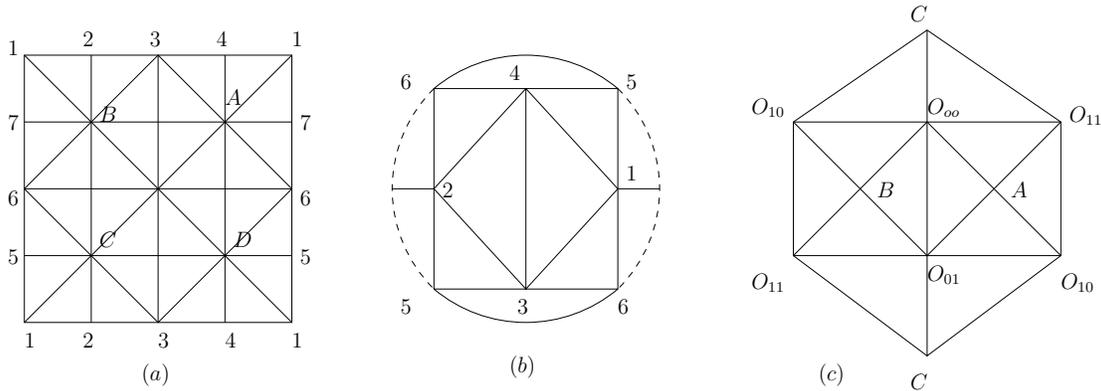
\begin{figure}[ht]
\centerline{
\scalebox{.70}{
\input{egch204.pstex_t}
 }
 }
\caption {Some equivariant triangulations of torus and $\mathbb{RP}^2$.}
\label{egch204}
\end{figure}
 \end{example}

\subsection{Triangulations of $3$-dimensional small covers}\label{sm3d}
Let $\xi: N \to Q$ be a small cover over a $3$-polytope $Q$ with $m$ facets $F_1, \ldots, F_m$
and with $k$ vertices $V_1, \ldots, V_k$. Let $(f_0, f_1, f_2)$ be the $f$-vector of $Q$, where $f_i$
is the number of codimension $i+1$ face of $Q$. We adhere the notations of Section \ref{pre}.
Let $I_i$ be the $3$-dimensional cube in the cubical subdivision $C(Q)$ corresponding to the vertex $V_i$
of $Q$ for $i=1, \ldots, k$. Let $C_F$ be the vertex of $C(Q)$ corresponding to the nonempty face $F$ of $Q$. So
$$\xi^{-1}(C_{Q}):= \{C_{000}, C_{100}, C_{010}, C_{001}, C_{110}, C_{101}, C_{011}, C_{111}\}.$$
In this subsection we assume that $a=C_{000}, b=C_{100}, c=C_{010}, d=C_{001}, e=C_{110}, f=C_{101}, g=C_{011}, h=C_{111}$.
Let $\beta : \mathcal{F}(Q) \to \ZZ^3_2$ be the $\ZZ_2$-characteristic function of $N$ and $\beta_i =\beta(F_i)$.
Let $H_i$ be the subgroup generated by $\beta_i$ of $\ZZ_2^3$ and $\{a_jH_i: j=1, 2, 3, 4\}$ be the set of
distinct cosets of $H_i$. So $a_jH_i= \{a_j, a_j+\beta_i\}$ for all $j=1,2,3,4$ and $i=1, \ldots, m$.
Then $\xi^{-1}(C_{Q}C_{F_i}) $ is the disjoint union of the edges $[C_{a_j}, C_{a_j+\beta_i}]$ for $1 \leq j \leq 4$,
where  $[C_{a_j}, C_{a_j+\beta_i}]$ is the edge with vertices $C_{a_j} $ and $C_{a_j+\beta_i}$ such that
$\xi([C_{a_j}, C_{a_j+\beta_i}])$ contains $C_{F_i}$.
For simplicity, suppose $V_1 =F_1 \cap F_2 \cap F_3$, $E_1= F_1 \cap F_2$, $E_2= F_2 \cap F_3$ and
$E_3 = F_3 \cap F_1$. Let $G_1$ be the subgroup generated by $\beta_1$ and $\beta_2$ of $\ZZ_2^3$ and $b_1G_1, b_2G_1$ are
distinct cosets of $G_1$ in $\ZZ_2^3$. Then $\xi^{-1}(C_{F_1}C_{E_1}C_{F_2}C_{Q})$ is the disjoint union of following
quadrangles, $$C_{b_j}C_{b_j+\beta_1}C_{b_j+\beta_2}C_{b_j+\beta_1+\beta_2}(C_{E_1}) \mbox { for } j=1,2.$$
Similarly for the sets $\xi^{-1}(C_{F_2}C_{E_2}C_{F_3}C_{Q})$ and $\xi^{-1}(C_{F_3}C_{E_3}C_{F_1}C_{Q})$. So $\xi^{-1}(I_{1})$
is a cube with vertices $\xi^{-1}(C_{Q})$ where the boundary is the union of six quadrangles given by $$\xi^{-1}(C_{F_1}C_{E_1}C_{F_2}C_{Q}),
\xi^{-1}(C_{F_2}C_{E_2}C_{F_3}C_{Q})~ \mbox{and} ~ \xi^{-1}(C_{F_3}C_{E_3}C_{F_1}C_{Q}).$$ For example see Figure \ref{egch203}.
Clearly $\cup_1^k \xi^{-1}(I_{i})$ is a cubical subdivision of $N$. Considering a triangulation for $\xi^{-1}(I_{i})$,
we may triangulate $N$. Even though $N$ is a union of these cubes, we can not consider any triangulation of these cubes.
The choice of triangulation for these cubes totally depends on $\ZZ_2^3$-action on $N$ and hence on the $\ZZ_2$-characteristic function,
see Examples \ref{n31}, \ref{n32} and \ref{n33}. 
Since fixed points of $\ZZ_2^3$-action on $N$, bijectively correspond to the vertices of $Q$, we also denote them
by $V_i$ for $i=1, \ldots, k$. 
Each facet of the cube $\partial{\xi^{-1}(I_i)}$ is a square. So any triangulation of $\partial(\xi^{-1}(I_{i}))$
contains at least $12$ triangles. Consider the cone on this triangulation with apex $V_i$. So we get a triangulation
of $\xi^{-1}(I_{i})$ with at least $12$ facets.

Let
\begin{equation}\label{sa}
s(a) = \# \{F_i : \beta(F_i) = a\}
\end{equation}
for all $a \in \ZZ_2^3- (0,0,0)$.
Let $F_{i_1}$ and $F_{i_2}$ be two facets such that $\beta_{i_1} = \beta_{i_2}$. So $\xi^{-1}(C_{F_{i_1}}C_{Q}C_{F_{i_2}})$
is a disjoint union of $4$ circles. Then the vertex set $\xi^{-1}(C_{F_{i_1}}C_{Q}C_{F_{i_2}}) \cap \{\xi^{-1}(C_{Q}) \cup
\{V_i: i=1, \ldots, k\}\}$ does not give  triangulations of these circles. Each of these circles contains two vertices of
$\xi^{-1}(C_{Q}) \cup \{V_i: i=1, \ldots, k\}$. So we need to add one vertex for each circle. We consider the corresponding
vertices from the set $\xi^{-1}(C_{F_{i_2}})$. Clearly the minimal triangulation of $\xi^{-1}(C_{F_{i_{1(a)}}}C_{Q} \cup
\cdots \cup C_{F_{i_{s(a)}}}C_{Q})$ needs $\xi^{-1}\{C_{Q}, C_{F_{i_{1(a)}}}, \ldots, C_{F_{i_{s(a)-1}}}\}$ vertices.
Then the number of vertices in $ \xi^{-1}\{C_{F_{i_{1(a)}}}, \ldots, C_{F_{i_{s(a)-1}}}\}$ is $4s(a) -4$.

 Let $E_j= F_{j_1} \cap F_{j_2}$ be an edge. Let $F_{j_3}$ and $F_{j_4}$
be the facets such that $V_{j_1}=F_{j_3} \cap E_j$ and $V_{j_2}=F_{j_4} \cap E_j$. Let $\beta_{j_3}+\beta_{j_4} \in
\{ \beta_{j_1}, \beta_{j_2}\}$. So the minimal triangulations for the boundaries of the cubes $\xi^{-1}(I_{j_1})$ and
$\xi^{-1}(I_{j_2})$ do not give a triangulation of the boundary of  $\xi^{-1}(I_{j_1} \cup I_{j_2})$.
Otherwise we have a $2$-edge triangulation of a circle, belongs to the boundary of $\xi^{-1}(I_{j_1} \cup I_{j_2})$.
We need to add two vertices corresponding to the point $C_{E_j}$ of the edge $E_j$. Let
\begin{equation}
  l= \#\{E_j : ~\mbox{the condition in the above paragraph holds for the edge}~ E_j\}.
\end{equation}
 So we add $2l$ many vertices corresponding to these edges. With these
collections of vertices we triangulate the boundary of each $\xi^{-1}(I_{i})$, then make cone at apex $V_i$. By construction,
this is an equilibrium triangulation. Thus we can prove the following.
\begin{theorem}\label{tri3sm}
Let $N$ be a $3$-dimensional small cover with characteristic pair $(Q, \xi)$. There is a natural equilibrium
triangulation of $N$ with $4(s(a_1)+\ldots + s(a_7))+2l+k-20$ vertices where $s(a_i)$'s are defined in Equation
$(\ref{sa})$ and $l$ is defined in the previous paragraph.
\end{theorem}

In the following examples we discuss some irreducible and vertex minimal equilibrium triangulations of some small
covers explicitly. We denote the characteristic pairs in Figure \ref{egch202} $(a)$, $(b)$, $(c)$ and $(d)$ by
$(\bigtriangleup^3, \beta)$, $(Q, \beta^{1})$, $(Q, \beta^2)$ and $(Q, \beta^3)$ respectively, where $Q$ is a $3$-dimensional
prism with vertices $1, \ldots, 6$. Let $\xi : N \to \bigtriangleup^3, \xi^1 : N_1 \to Q$, $\xi^2: N_2 \to Q$
and $\xi^3 : N_3 \to Q$ be the respective small covers.

\begin{figure}[ht]
\centerline{
\scalebox{0.60}{
\input{egch205.pstex_t}
 }
 }
\caption {Some $\ZZ_2$-characteristic functions of a tetrahedron and prism.}
\label{egch202}
\end{figure}
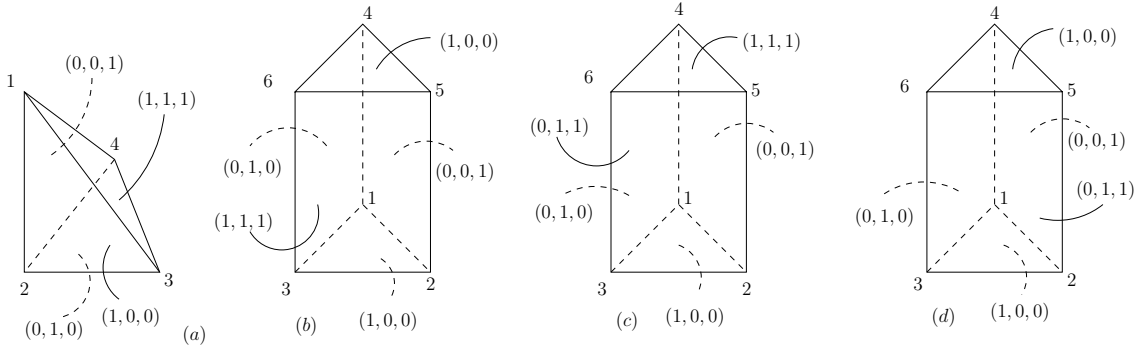

\begin{example}\label{trirp3}
Let $\xi : N \to \bigtriangleup^3$ be a small cover over the $3$-simplex with vertices $1, \ldots, 4$.
So $N$ is $\mathbb{RP}^3$. The $\ZZ_2$-characteristic function of $\mathbb{RP}^3$ is given in Figure
\ref{egch202} $(a)$.

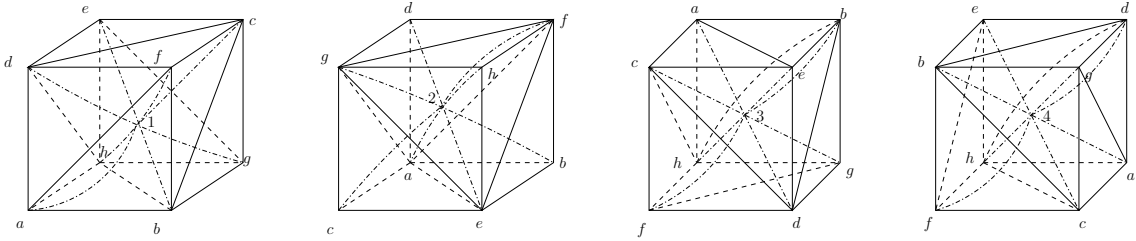
\begin{figure}[ht]
\centerline{
\scalebox{0.50}{
\input{egch203.pstex_t}
 }
 }
\caption {Some triangulation of equivariantly diffeomorphic cubes of $\xi^{-1}(I_1)$, $\xi^{-1}(I_2)$, $\xi^{-1}(I_3)$ and
$\xi^{-1}(I_4)$ respectively.}

\label{egch203}
\end{figure} 
The triangulations in Figure \ref{egch203}, give a natural $12$ vertex triangulation of $\mathbb{RP}^3$ with
the vertices $1, 2, 3, 4, a, \ldots, h$. Note that this is an equilibrium but not $\ZZ_2^3$-equivariant
triangulation of $\mathbb{RP}^3$.
\end{example}
\begin{lemma}
 There is an equilibrium triangulation of $\mathbb{RP}^3$ with $11$ vertices.  
\end{lemma}
\begin{proof}
We delete the vertex $2$ from the triangulation of $\mathbb{RP}^3$ in Example \ref{trirp3}.
Then consider the triangulation of $\xi^{-1}(I_2)$ with five $3$-simplices.
This gives an equilibrium triangulation of $\mathbb{RP}^3$ with $11$ vertices. Since any triangulation
of $\mathbb{RP}^3$ contains at least $11$ vertices, this is a vertex minimal equilibrium triangulation
of $\mathbb{RP}^3$.
\end{proof}

\begin{example}[Triangulation for $N_1$]\label{n31}
The manifold $N_1$ is $\mathbb{RP}^3 \# {\mathbb{RP}^3}$. Let $(\xi^1)^{-1}(i)=i$, $i=1, \ldots, 6$.
In this case we have six cubes $(\xi^1)^{-1}(I_{i})$ for $i=1, \ldots, 6$ in the cubical subdivision of
$N_1$. For the facets $F_{i_1}=123$ and $F_{i_2} = 456$ we have $\beta^1_{i_1} = \beta^1_{i_2}= (1,0,0)$.
Since $[d, f]$, $[a,b]$, $[c, e]$ and $[g, h]$ are
edges of first $3$ cubes, they can not form edges in the remaining $3$ cubes. Otherwise we have a
$2$-edge triangulation of a circle. Hence we need to add $p, q, r $ and $s$ in the interiors of those
edges respectively. Now we can construct a triangulation of the boundary of the cubes $(\xi^1)^{-1}(I_{i})$
such that they are compatible on the boundaries, see Figure \ref{egch207}. Then taking cone on the boundary
of $(\xi^1)^{-1}(I_{i})$ from $i$ for $i=1, 2, 4, \ldots, 6$ we get an equilibrium triangulation of $N_1$
with $17$ vertices.
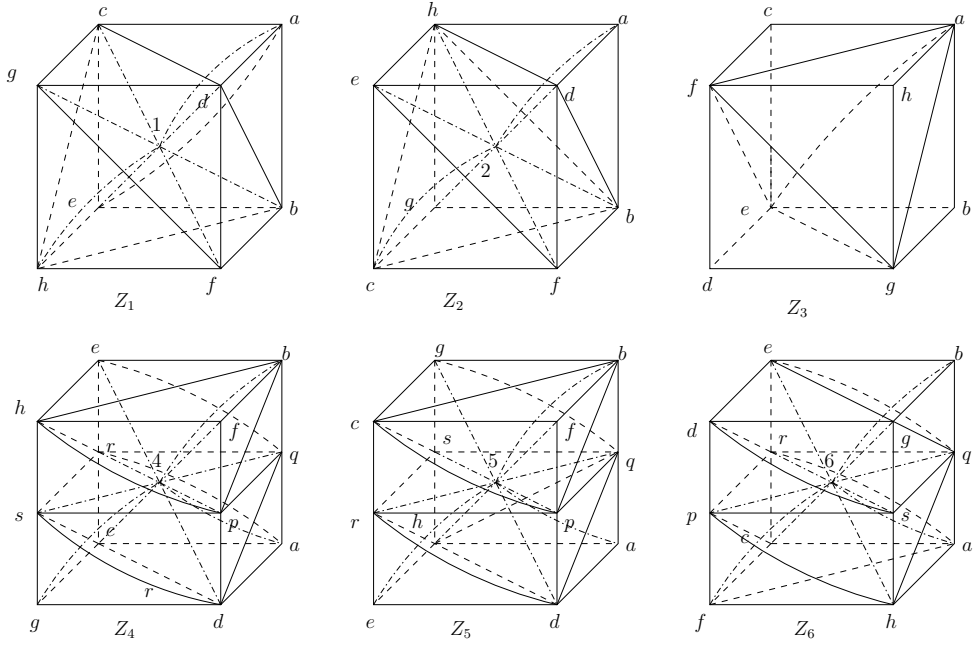
\begin{figure}[ht]
\centerline{
\scalebox{0.64}{
\input{egch201c.pstex_t}
 }
 }
\caption {Triangulations of  zones of influences for $N_1$.}
\label{egch207}
\end{figure}
We will show that this is a vertex minimal equilibrium triangulation of $\mathbb{RP}^3 \# \mathbb{RP}^3$.
The $f$-vector of this simplicial complex is $(f_0,f_1,f_2,f_3)= (17, 106, 178, 89)$. Note that the
minimal triangulations of $\mathbb{RP}^3 \# \mathbb{RP}^3$ contain $15$ vertices. 
\end{example}

\begin{theorem}\label{rp33}
Let $X$ be an equilibrium triangulation of $\mathbb{RP}^3 \# \mathbb{RP}^3$ with complete zones of
influences $Z_1, \ldots, Z_6$. If $V(X) \cap Z_i$ is the equilibrium set for some $i \in \{1, \ldots,
6\}$, then $X$ contains at least $17$ vertices.
\end{theorem}
\begin{proof}
 We may assume $Z_i = (\xi^1)^{-1}(I_i)$ for $i=1, \ldots, 6$ and $V(X) \cap Z_3$ is the equilibrium
set $\{7, \ldots, 14\}$. So $Z_i \cap X$ is a triangulation of $Z_i$ for $i=1, \ldots, 6$ and $\{a, \ldots, h\} \subset V(X)$.
We know that without extra vertices a $3$-cube has $7$ distinct triangulations. One with
$5$ $3$-dimensional simplices and others with $6$ $3$-dimensional simplices. If we consider the triangulation of $Z_3$
with $6$ $3$-dimensional simplices, there will be a diagonal edge in $Z_3$. So we can not get any triangulations
of $Z_1$ and $Z_2$ such that the boundary triangulations of these three cubes $Z_1, Z_2$ and $Z_3$
are compatible. Let us consider the triangulation of $Z_3$ with $5$ $3$-dimensional simplices. Considering this
triangulation of $Z_3$, we get a unique triangulation of the boundary of $Z_1$ and $Z_2$, see
Figure \ref{egch207}. These boundary triangulations of $Z_1$ and $Z_2$ do not give any
triangulations of them without extra vertices. So we need to consider at least one vertex
in each of their interiors.

By construction of $N_1$, we get that the intersections $[d, f] \cap Z_4 \cap V(X),
[a, b] \cap Z_4 \cap V(X), [c, e] \cap Z_4 \cap V(X) ~\mbox{and} ~ [g, h] \cap Z_4 \cap V(X)$
are nonempty. Suppose each of these intersections contains $p, q, r$ and $s$ respectively.
From the boundary triangulations of $Z_1, Z_2$ and $Z_3$, we get that any one of $[d, h],
[g,f], [b, d], [a, f], [a, e], [b, c], [c, h]$ $ ~\mbox{and}~[e, g]$ can not be an edge
of the cube $Z_4$. Similarly, any one of $[c, d], [e, f], [b, d], [a, f], [a, g],
[b, h], [c, h] $ $ ~\mbox{and} ~ [e, g]$ can not be an edge of the cube $Z_5$. Also, any one of
$[d,h], [f, g], [c, d], [e, f], [d, e], $ $[b, c], [a, g]~ \mbox{and} ~[b, h]$ can
not be an edge of the cube $Z_6$. So $\{[p, q], [q,r], [r, s], [p, s] \},$ $
\{[p,q], [q, s], [r, s], [p, r]\}$ $\mbox{and} ~ \{[p, s],[q, s], [q, r], [p, r]\}$
are sets of edges in the cubes $Z_4, Z_5$ and $Z_6$ respectively. So, neither $[p , r]$ nor $[q, s]$ can be
an edge in the cube $Z_4$, since they are already edges in the cube $Z_5$. For similar
reason neither $[q, r]$ nor $[p, s]$ is an edge of the cube $Z_5$ and neither $[p, q]$
nor $[r, s]$ is an edge in the cube $Z_6$.

Also, none of $[d,e], [a, h], [c, f]$ and $[b, g]$ can be an edge in the cube $Z_4$,
since they are already edges in the cube $Z_2$. For similar reason, none of $[d, g], [a, c]$,
$[b, f]$, $[b, e]$ can be an edge of the cube $Z_5$ and none of $[b, f]$, $[a, d]$,
$[c, g]$, $[e, h]$ can be an edge of the cube $Z_6$. So the rectangles with vertices
$p, q, r, s$ in $Z_4, Z_5$ and $Z_6$ contain at least one vertex of $X$
other than its vertices. Let $x_1, x_2$ and $x_3$ be the vertices which belong to these rectangles
respectively. Since the intersection of these rectangles is $\{p, q, r, s\}$, $x_1, x_2$ and
$x_3$ can not be equal. So $|\{x_1, x_2, x_3\}| \geq 2$. Suppose $x_1=x_2$ and $x_1$ belongs
to the relative interior of $[p, q]$. Consider the cone from $x_1$ on the boundary of $Z_4$
and $Z_5$. Then the triangle $[x_1,r, s]$ will be face of $4$ distinct $3$-simplices. So,
$x_1, x_2$ and $x_3$ are distinct. Hence $|V(X)| \geq 17$.

Let $1$ and $2$ be points belong to the interiors of $Z_1$ and $Z_2$ respectively. Then take
cone on the boundary triangulation of $Z_1$ and $Z_2$ from $1$ and $2$ respectively. Let $x_1,
x_2$ and $x_3$ be the interior points of the cubes $Z_4, Z_5$ and $Z_6$ respectively. Consider
the cone on the boundary of $Z_4, Z_5$ and $Z_6$ from $x_1, x_2$ and $x_3$ respectively.
Without loss, we can assume $(x_1, x_2, x_3) = (4, 5, 6)$. Then,
we can construct some triangulations of $Z_4, Z_5$ and $Z_6$ such that they are compatible
on the boundary. Hence we get some $17$ vertex equilibrium triangulations of $\mathbb{RP}^3
\# \mathbb{RP}^3$.
\end{proof}

\begin{example}[Triangulation for $N_2$]\label{n32}
The manifold $N_2$ is a non-trivial $S^1$ bundle over $\mathbb{RP}^2$. Let $(\xi^2)^{-1}(i)=i$
for $i=1, \ldots, 6$. In this case we have six cubes $(\xi^2)^{-1}(I_{i})$ for $i=1, \ldots, 6$
in the cubical subdivision of $N_2$. Let $E_i= F_{i_1} \cap F_{i_2}$ be an edge belongs to the
triangular facet whose vertices are $i_1=F_{i_3} \cap E_i$ and $i_2=F_{i_4} \cap E_i$. Observe
that $\beta^2_{i_3}+\beta^2_{i_4} \in \{ \beta^2_{i_1}, \beta^2_{i_2}\}$ if $(i_1, i_2) \in
\{(1, 3), (2, 3), (2, 5), (4, 6), (5, 6)\}$. So the minimal triangulation for the boundary of
the cubes $(\xi^2)^{-1}(I_{i_1})$ and $\xi^{-1}(I_{i_2})$ may not give a triangulation of the
boundary of $(\xi^2)^{-1}(I_{i_1} \cup I_{i_2})$. Otherwise we have a $2$-edge triangulation
of a circle, belongs to the boundary of $(\xi^2)^{-1}(I_{i_1} \cup I_{i_2})$. So we need to
add two vertices corresponding to the point $C_{E_i}$ in $C(Q)$ of the edge $E_i$. For any two
facets $F_{i_1}$ and $F_{i_2}$ we have $\beta^2_{i_1} \neq \beta^2_{i_2}$. For simplicity of
notation, the vertices corresponding to $C_{E_i}$'s are denoted by $q, r, s, t, u, v, w, x,
y, z$ accordingly, see Figure \ref{egch208}. Now we can construct a triangulation of $N_2$ by
taking cone from $i$ on the triangulation of the boundary (given in Figure \ref{egch208}) of
the cube $(\xi^2)^{-1}(I_{i})$ for $i=1, \ldots, 6$. Clearly, this is an equilibrium triangulation.
\begin{figure}[ht]
\centerline{
\scalebox{0.64}{
\input{egch208.pstex_t}
 }
 }
\caption {Triangulations of  zones of influences for $N_2$.}
\label{egch208}
\end{figure}
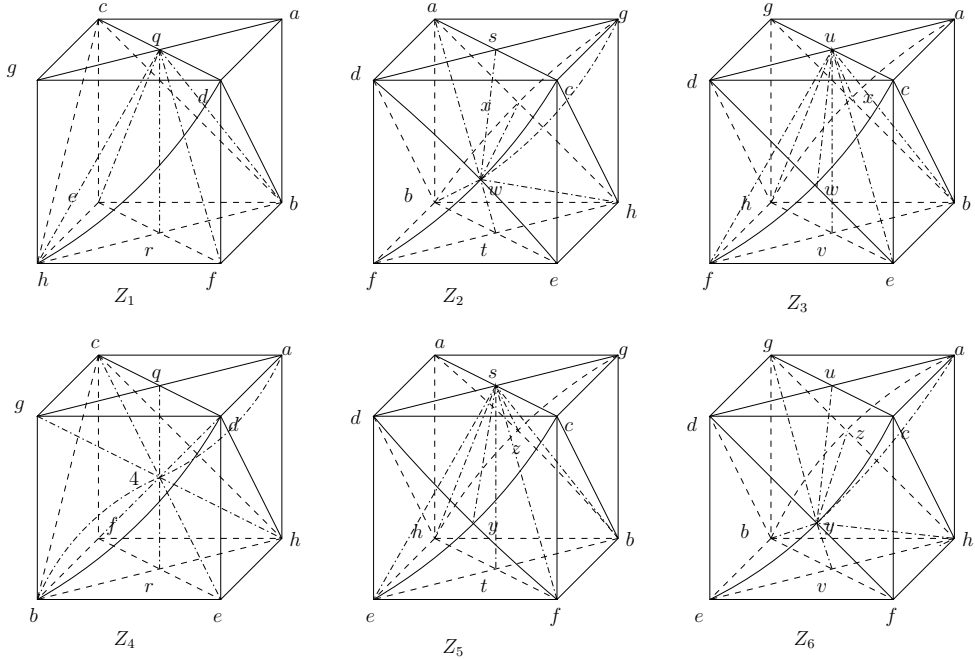 
Observing this triangulation, we can construct some $19$-vertex equilibrium triangulation of $N_2$ using
the vertex set $\{4, a, \ldots, h, q, \ldots, z\}$, see Figure \ref{egch208}. The $f$-vector of this simplicial
complex is $(f_0,f_1,f_2,f_3)=(19, 111,184, 92)$.
\end{example}
 
\begin{theorem}\label{nn32}
Any equilibrium triangulation of $N_2$ contains at least $19$ vertices. There are some $19$
vertex equilibrium triangulations of $ N_2$.
\end{theorem}
\begin{proof}
Let $X$ be an equilibrium triangulation of $N_2$ with complete zones of influences $Z_1,
\ldots,$ $ Z_6$.  We may assume $Z_i = (\xi^2)^{-1}(I_i)$ for $i=1, \ldots, 6$. So $Z_i \cap X$
is a triangulation of $Z_i$ for $i=1, \ldots, 6$ and $\{a, \ldots, h\} \subset V(X) $ is
the equilibrium set.

Observe that neither $[c, d]$ nor $[a, g]$ can be an edge in the cube $Z_1$, since they
are already edges in the cube $Z_2$. For similar reason neither $[d, g]$ nor $[a, c]$ is
an edge of the cube $Z_2$ and neither $[a, d]$ nor $[c, g]$ is an edge of the cube $Z_3$.
Let $A_1, A_2$ and $A_3$ be the facets of $Z_1, Z_2$ and $Z_3$ with vertex set $\{a, c, d, g\}$
respectively. So $A_i$ contains at least one vertex of $X$ other than its vertices for $i=1,2,3$.
Suppose $A_i$ contains only $x_i \in V(X)$ for $i=1,2,3$. Since the intersection $A_1 \cap A_2
\cap A_3 = \{a, c, d, g\}$, $x_1, x_2, x_3$ can not be equal. Suppose $x_1 = x_2$ and $x_1$
belongs to the relative interior of $[a, d]$ in $Z_1$ and $Z_2$. Then $x_3$ can not belong
to the boundary of $A_3$, otherwise we get $2$-vertex triangulation of a circle in $X$. Also for the same
reason $x_1$ can not belong to the boundary of $A_1$ and $A_2$. So $x_i$ belongs to
interior of $A_i$ for $i=1,2,3$. 

Similarly, we can show the interior of the each facets of $\partial (Z_1), \partial (Z_2)$ and $\partial (Z_3)$ with
vertex set $\{b, e, f, h\}$ contains at least one vertex of $X$. Again by similar arguments,
we can show the interior of the each facets of $Z_2$ and $Z_5$ with vertex set $\{c, d, e, f\}$
(resp., $\{a, b, g,h\}$) contains at least one vertex of $X$. Since diagonal edges of $Z_1$ and
$Z_4$ can not be edges, any triangulation of $Z_1$ with the vertices as in Figure \ref{egch208}
contains the edge $[q, r]$ and $Z_4$ contains a vertex (say $4 \in V(X)$) in the interior.
Hence $|V(X)| \geq 19$. Considering these $19$ vertices we can construct some equilibrium
triangulations of $N_2$, see Figure \ref{egch208}. Hence we get the theorem.
\end{proof}

\begin{example}[Triangulation for $N_3$]\label{n33}
The manifold $N_3$ is $ S^1 \times \mathbb{RP}^2$. Let $(\xi^3)^{-1}(i)=i$, $i=1, \ldots, 6$.
In this case we have six cubes $Z_i=(\xi^3)^{-1}(I_{i})$ for $i=1, \ldots, 6$. For this
$\ZZ_2$-characteristic pair $(Q, \beta^3)$, there are two facets $123$ and $456$ with
characteristic vector $(1, 0, 0)$. So, considering the minimal triangulation of the boundary
of these $6$ cubes, we may not construct a triangulation of $N_3$. Since $[d,f]$, $[a, b]$,
$[c, e]$ and $[g, h]$ are also edges of the cubes $Z_1, Z_2, Z_3$, they can not form edge
in the remaining $3$ cubes. Otherwise we have a $2$-edge triangulation of a circle. Hence we
need to add $q, r, s$ and $t$ in the interiors of those edges in the last $3$ cubes respectively.
Since characteristic vectors of $123$ and $456$ are same, it is enough to consider
vertices corresponding to the edges $[1, 2], [2, 3]$ and $[1, 3]$ of $Q_3$. We denote
these pair of vertices by $\{u, v\}$, $\{w, x\}$ and $\{y, z\}$ respectively. Now we can
construct a triangulation of the cubes $(\xi^3)^{-1}(I_i)$ such that their boundaries
are compatible and each of $1$-faces belongs to exactly two $2$-faces of the boundaries,
see Figure \ref{egch201}. So we get an equilibrium triangulation of $N_3$.
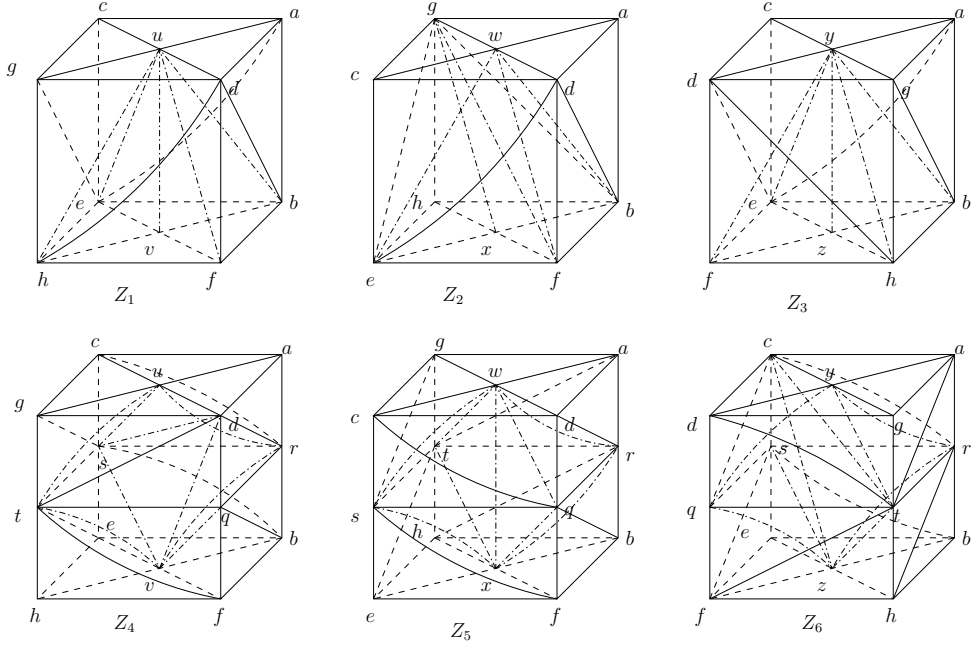
\begin{figure}[ht]
\centerline{
\scalebox{0.64}{
\input{egch201.pstex_t}
 }
 }
\caption {Triangulations of  zones of influences for $N_3$.}
\label{egch201}
\end{figure} 
From this triangulation we can construct an equilibrium triangulation of $\mathbb{RP}^2 \times S^1$
with $18$ vertices, see Figure \ref{egch201}. The $f$-vector of this complex is $(f_0,f_1,f_2,f_3)
=(18, 114, 192, 96)$. Note that the minimal triangulations of $\mathbb{RP}^2 \times S^1$ contain $14$ vertices.
\end{example}

\begin{theorem}\label{s1rp2}
Any equilibrium triangulation of $S^1 \times \mathbb{RP}^2$ contains at least $18$ vertices.
There are some $18$ vertex equilibrium triangulations of $ S^1 \times \mathbb{RP}^2$.
\end{theorem}
\begin{proof}
Let $X$ be an equilibrium triangulation of $S^1 \times \mathbb{RP}^2$ with complete zones of
influences $Z_1, \ldots, Z_6$.  We may assume $Z_i = (\xi^3)^{-1}(I_i)$ for $i=1, \ldots, 6$.
So $Z_i \cap X$ is a triangulation of $Z_i$ for $i=1, \ldots, 6$ and $\{a, \ldots, h\}
\subset V(X) $ is the equilibrium set. By similar arguments as in Example \ref{n33}, we get
$q, r, s, t \in (V(X) \cap Z_i)$ for $i=4,5,6$.

Now, neither $[c, d]$ nor $[a, g]$ can be an edge in the cube $Z_1$, since they are already edges
in the cube $Z_2$. For similar reason neither $[d, g]$ nor $[a, c]$ is an edge of the cube
$Z_2$ and neither $[a, d]$ nor $[c, g]$ is  an edge of the cube $Z_3$. Let $A_1, A_2$ and $A_3$
be the rectangles with vertex set $\{a, c, d, g\}$ in $Z_1, Z_2$ and $Z_3$ respectively. So $A_i$
contains at least one vertex of $X$ other than its vertices for $i=1,2,3$. Suppose $A_i$ contains
only $x_i \in V(X)$ for $i=1,2,3$. Since the intersection $A_1 \cap A_2 \cap A_3 = \{a, c, d, g\}$,
$x_1, x_2, x_3$ can not be equal. Suppose $x_1=x_2$ and $x_1$ belongs to the relative
interior of $[a, d]$ in $Z_1$ and $Z_2$. Then $x_3$ can not belong to the boundary of $A_3$,
otherwise we get $2$-vertex triangulation of a circle in $X$. Also for the same reason $x_1$ can not belong to
the boundary of $A_1$ and $A_2$. So $x_i$ belongs to interior of $A_i$ for $i=1,2,3$. Hence
$A_1 \cap A_2 \cap A_3 \cup V(X)$ contains at least $3$ vertices.
Similarly, we can show the interiors of the rectangles with vertex set $\{b, e, h, f\}$ in the cubes $Z_1,
Z_2$ and $Z_3$ contain at least one vertex of $X$. Hence $|V(X)| \geq 18$. Considering these $18$
vertices we can construct some equilibrium triangulations of $N_3$, see Figure \ref{egch201}.

\end{proof}

\subsection{Triangulations of real projective spaces}\label{triproj}
The vertex minimal triangulation of $\mathbb{RP}^n$ is not known for $n > 5$. It is known that $\mathbb{RP}^n$
has triangulation with at least $2^{n+1} -1$ vertices when $n>5$, (cf. \cite{[Da]}).
A minimal triangulation of $\mathbb{RP}^3$ was constructed by Walkup with $11$ vertices, cf. \cite{[Wa]}.
In 1999 and 2005, F. H. Lutz constructed one $16$- and $24$-vertex triangulations of $\mathbb{RP}^4$
and $\mathbb{RP}^5$ respectively, see \cite{[Lu]}. Recently, the second author found a $(\frac{n(n+5)}{2} +1)$-vertex
non equilibrium triangulation of $\mathbb{RP}^n$ for $n \geq 3$, which is not 
far from the theoretical optimum $\frac{(n+1)(n+2)}{2} +1$, see \cite{[Sa]}. Here we present the following.
\begin{theorem}\label{tripro}
The real projective space $\mathbb{RP}^n$ has an equilibrium triangulation with $2^n+n+1$ vertices for all $n \geq 2$.
\end{theorem}
\begin{proof}
Let $C(\bigtriangleup^n)$ be a cubical subdivision of the $n$-dimensional simplex $\bigtriangleup^n$. Let
$V_1, \ldots, V_n,$ $ V_{n+1}$ be the vertices of $\bigtriangleup^n$.
For $\bigtriangleup^n$ and any proper face $F$ of $\bigtriangleup^n$, let $O_{\bigtriangleup^n}$ and $C_F$ be the corresponding points in
$C(\bigtriangleup^n)$ respectively. Let $\xi : \mathbb{RP}^n \to \bigtriangleup^n$
be the orbit map. Recall the construction of $I^n_{F_1 \subseteq F_2}$ from  cubical subdivision in page-51 of \cite{[BP]}.
Clearly $I^n_{{V_i}\subset \bigtriangleup^n}= I^n_i$ for all vertex $V_i$ of $\bigtriangleup^n$. Let $F$ be a $k$-dimensional face
of $\bigtriangleup^n$. Then $I^n_{F \subset \bigtriangleup^n}$ is an $(n-k)$-dimensional face of $I^n_i$ for all $V_i \in F$.
Since the action of $\ZZ_2^n$ on $\mathbb{RP}^n$ is locally standard, the subset $\xi^{-1}(I^n_i)$ of $\mathbb{RP}^n$
is diffeomorphic as manifold with corners to an $n$-dimensional cube $A^n_i$ with vertices $\xi^{-1}(O_{\bigtriangleup^n})$.
 The point $\xi^{-1}(V_i)$ belongs
to the interior of $A^n_i$. Using the cubical subcomplex of $\bigtriangleup^n$ we construct a cell complex of $\mathbb{RP}^n$
in the following way. Let $X_0=\{\xi^{-1}(V_i)(:=V_i^{\prime}) : i = 1, \ldots, n+1\} \cup \{\xi^{-1}(O_{\bigtriangleup^n})\}$.
So the number of points in $X_0$ is $2^n + n+1$. The cone on $ I^n_{F \subset \bigtriangleup^n}$ with apex $V_i$ in
$\bigtriangleup^n$ is denoted by $V_iI^n_{F \subset \bigtriangleup^n}$ for any face $F$ containing $V_i$.  Let
$X_1= \{\xi^{-1}(I^n_{F \subset \bigtriangleup^n}) : F ~\mbox{is an $(n-1)$-dimension face of} ~\bigtriangleup^n\}$
$\cup \{\xi^{-1}(V_iO_{\bigtriangleup^n}): i=1, \ldots, n\}$, where $V_iO_{\bigtriangleup^n}$ line segment joining
$V_i$ and $O_{\bigtriangleup^n}$ in $\bigtriangleup^n$. Let $X_2= \{\xi^{-1}(I^n_{F \subset \bigtriangleup^n}) :
F ~\mbox{is an $(n-2)$-dimension face of}~ \bigtriangleup^n\}$ $\cup \{V_i^{\prime}\xi^{-1}(I^n_{F \subset \bigtriangleup^n} :
i=1, \ldots, n) :$ $ F ~\mbox{is an $(n-1)$-dimension face of}~ \bigtriangleup^n\}$. Continue this process up to
$n$th step $X_n$ where $X_n= \{V_i^{\prime}\xi^{-1}(I^n_{V_i \subset \bigtriangleup^n}) : i=1, \dots, n\} \cup \{\xi^{-1}(I^n_{V_{n+1}})\}$.
Then $X = \cup_{i=0}^n X_i$ gives a cell complex of $\mathbb{RP}^n$. Note that each $X_i$ is a cell complex for
$i=0, \ldots, n$. So by Proposition \ref{thmrs} we can construct nice triangulations of
$X$ without adding any new vertices. Clearly, this is an equilibrium triangulation of $\mathbb{RP}^n$
and the number of vertices of the triangulation of $X$ is $2^n+n+1$.
\end{proof}

\begin{remark}
 An equilibrium triangulation of a small cover $N^n$ may not be a $\ZZ^n_2$-equivariant triangulation and 
a $\ZZ^n_2$-equivariant triangulation of a small cover $N^n$ may not be an equilibrium triangulation.
\end{remark}

{\bf Acknowledgement.} The authors express their gratitude to Basudeb Datta and Mainak Poddar
for helpful suggestions. The authors thank anonymous referees for many helpful comments.
The first author would like to thank Council of Scientific and Industrial Research
(Shyamaprasad Mukherjee Fellowship) and UGC Center for Advanced Study, India for financial
support. The second author would like to thank Indian Institute of Science and Korea Advanced
Institute of Science and Technology for supporting him with the post doctoral fellowship.


\vspace{1cm}

\vfill
\end{document}

%% file: egch201a.pstex_t
\begin{picture}(0,0)%
\includegraphics{egch201a.pstex}%
\end{picture}%
\setlength{\unitlength}{4144sp}%
\begingroup\makeatletter\ifx\SetFigFont\undefined%
\gdef\SetFigFont#1#2#3#4#5{%
  \reset@font\fontsize{#1}{#2pt}%
  \fontfamily{#3}\fontseries{#4}\fontshape{#5}%
  \selectfont}%
\fi\endgroup%
\begin{picture}(8438,1893)(796,-2875)
\put(1756,-1141){\makebox(0,0)[lb]{\smash{{\SetFigFont{12}{14.4}{\rmdefault}{\mddefault}{\updefault}{\color[rgb]{0,0,0}$v_2$}%
}}}}
\put(2521,-1951){\makebox(0,0)[lb]{\smash{{\SetFigFont{12}{14.4}{\rmdefault}{\mddefault}{\updefault}{\color[rgb]{0,0,0}$v_1$}%
}}}}
\put(811,-1951){\makebox(0,0)[lb]{\smash{{\SetFigFont{12}{14.4}{\rmdefault}{\mddefault}{\updefault}{\color[rgb]{0,0,0}$v_3$}%
}}}}
\put(1666,-2761){\makebox(0,0)[lb]{\smash{{\SetFigFont{12}{14.4}{\rmdefault}{\mddefault}{\updefault}{\color[rgb]{0,0,0}$v_4$}%
}}}}
\put(3961,-2761){\makebox(0,0)[lb]{\smash{{\SetFigFont{12}{14.4}{\rmdefault}{\mddefault}{\updefault}{\color[rgb]{0,0,0}$v_4$}%
}}}}
\put(4771,-1951){\makebox(0,0)[lb]{\smash{{\SetFigFont{12}{14.4}{\rmdefault}{\mddefault}{\updefault}{\color[rgb]{0,0,0}$v_1$}%
}}}}
\put(3916,-1141){\makebox(0,0)[lb]{\smash{{\SetFigFont{12}{14.4}{\rmdefault}{\mddefault}{\updefault}{\color[rgb]{0,0,0}  $v_2$}%
}}}}
\put(3106,-1996){\makebox(0,0)[lb]{\smash{{\SetFigFont{12}{14.4}{\rmdefault}{\mddefault}{\updefault}{\color[rgb]{0,0,0}$v_3$}%
}}}}
\put(7021,-1951){\makebox(0,0)[lb]{\smash{{\SetFigFont{12}{14.4}{\rmdefault}{\mddefault}{\updefault}{\color[rgb]{0,0,0}$v_1$}%
}}}}
\put(6256,-1141){\makebox(0,0)[lb]{\smash{{\SetFigFont{12}{14.4}{\rmdefault}{\mddefault}{\updefault}{\color[rgb]{0,0,0}$v_2$}%
}}}}
\put(5356,-1906){\makebox(0,0)[lb]{\smash{{\SetFigFont{12}{14.4}{\rmdefault}{\mddefault}{\updefault}{\color[rgb]{0,0,0}$v_3$}%
}}}}
\put(6166,-2761){\makebox(0,0)[lb]{\smash{{\SetFigFont{12}{14.4}{\rmdefault}{\mddefault}{\updefault}{\color[rgb]{0,0,0}$v_4$}%
}}}}
\put(9091,-2581){\makebox(0,0)[lb]{\smash{{\SetFigFont{12}{14.4}{\rmdefault}{\mddefault}{\updefault}{\color[rgb]{0,0,0}$v_1$}%
}}}}
\put(9001,-1321){\makebox(0,0)[lb]{\smash{{\SetFigFont{12}{14.4}{\rmdefault}{\mddefault}{\updefault}{\color[rgb]{0,0,0}$v_2$}%
}}}}
\put(7831,-1366){\makebox(0,0)[lb]{\smash{{\SetFigFont{12}{14.4}{\rmdefault}{\mddefault}{\updefault}{\color[rgb]{0,0,0}$v_3$}%
}}}}
\put(7876,-2536){\makebox(0,0)[lb]{\smash{{\SetFigFont{12}{14.4}{\rmdefault}{\mddefault}{\updefault}{\color[rgb]{0,0,0}$v_4$}%
}}}}
\put(6346,-1816){\makebox(0,0)[lb]{\smash{{\SetFigFont{12}{14.4}{\rmdefault}{\mddefault}{\updefault}{\color[rgb]{0,0,0}$v_0$}%
}}}}
\put(8641,-1996){\makebox(0,0)[lb]{\smash{{\SetFigFont{12}{14.4}{\rmdefault}{\mddefault}{\updefault}{\color[rgb]{0,0,0}$v_0$}%
}}}}
\put(2431,-2626){\makebox(0,0)[lb]{\smash{{\SetFigFont{12}{14.4}{\rmdefault}{\mddefault}{\updefault}{\color[rgb]{0,0,0}$(1)$}%
}}}}
\put(4591,-2626){\makebox(0,0)[lb]{\smash{{\SetFigFont{12}{14.4}{\rmdefault}{\mddefault}{\updefault}{\color[rgb]{0,0,0}$(2)$}%
}}}}
\put(6841,-2626){\makebox(0,0)[lb]{\smash{{\SetFigFont{12}{14.4}{\rmdefault}{\mddefault}{\updefault}{\color[rgb]{0,0,0}$(3)$}%
}}}}
\put(8686,-2806){\makebox(0,0)[lb]{\smash{{\SetFigFont{12}{14.4}{\rmdefault}{\mddefault}{\updefault}{\color[rgb]{0,0,0}$(4)$}%
}}}}
\end{picture}%

%% file: egsm1.pstex_t
\begin{picture}(0,0)%
\includegraphics{egsm1.pstex}%
\end{picture}%
\setlength{\unitlength}{4144sp}%
\begingroup\makeatletter\ifx\SetFigFont\undefined%
\gdef\SetFigFont#1#2#3#4#5{%
  \reset@font\fontsize{#1}{#2pt}%
  \fontfamily{#3}\fontseries{#4}\fontshape{#5}%
  \selectfont}%
\fi\endgroup%
\begin{picture}(5700,2269)(661,-3014)
\put(1846,-2041){\makebox(0,0)[lb]{\smash{{\SetFigFont{12}{14.4}{\rmdefault}{\mddefault}{\updefault}{\color[rgb]{0,0,0}$V_j$}%
}}}}
\put(2746,-2041){\makebox(0,0)[lb]{\smash{{\SetFigFont{12}{14.4}{\rmdefault}{\mddefault}{\updefault}{\color[rgb]{0,0,0}$V_{j+1}$}%
}}}}
\put(1801,-916){\makebox(0,0)[lb]{\smash{{\SetFigFont{12}{14.4}{\rmdefault}{\mddefault}{\updefault}{\color[rgb]{0,0,0}$V_{j-1}$}%
}}}}
\put(676,-2131){\makebox(0,0)[lb]{\smash{{\SetFigFont{12}{14.4}{\rmdefault}{\mddefault}{\updefault}{\color[rgb]{0,0,0}$V_{j+1}$}%
}}}}
\put(1846,-2941){\makebox(0,0)[lb]{\smash{{\SetFigFont{12}{14.4}{\rmdefault}{\mddefault}{\updefault}{\color[rgb]{0,0,0}$V_{j-1}$}%
}}}}
\put(1846,-1681){\makebox(0,0)[lb]{\smash{{\SetFigFont{12}{14.4}{\rmdefault}{\mddefault}{\updefault}{\color[rgb]{0,0,0}$e$}%
}}}}
\put(5536,-1636){\makebox(0,0)[lb]{\smash{{\SetFigFont{12}{14.4}{\rmdefault}{\mddefault}{\updefault}{\color[rgb]{0,0,0}$e$}%
}}}}
\put(5491,-2311){\makebox(0,0)[lb]{\smash{{\SetFigFont{12}{14.4}{\rmdefault}{\mddefault}{\updefault}{\color[rgb]{0,0,0}$\beta_j(e)$}%
}}}}
\put(6346,-1996){\makebox(0,0)[lb]{\smash{{\SetFigFont{12}{14.4}{\rmdefault}{\mddefault}{\updefault}{\color[rgb]{0,0,0}$V_{j+1}$}%
}}}}
\put(5446,-961){\makebox(0,0)[lb]{\smash{{\SetFigFont{12}{14.4}{\rmdefault}{\mddefault}{\updefault}{\color[rgb]{0,0,0}$V_{j-1}$}%
}}}}
\put(1531,-1591){\makebox(0,0)[lb]{\smash{{\SetFigFont{12}{14.4}{\rmdefault}{\mddefault}{\updefault}{\color[rgb]{0,0,0}$\sigma$}%
}}}}
\end{picture}%

%% file: egch201b.pstex_t
\begin{picture}(0,0)%
\includegraphics{egch201b.pstex}%
\end{picture}%
\setlength{\unitlength}{3947sp}%
\begingroup\makeatletter\ifx\SetFigFont\undefined%
\gdef\SetFigFont#1#2#3#4#5{%
  \reset@font\fontsize{#1}{#2pt}%
  \fontfamily{#3}\fontseries{#4}\fontshape{#5}%
  \selectfont}%
\fi\endgroup%
\begin{picture}(9180,1506)(1111,-3055)
\put(1951,-2986){\makebox(0,0)[lb]{\smash{{\SetFigFont{12}{14.4}{\rmdefault}{\mddefault}{\updefault}{\color[rgb]{0,0,0}$v_1$}%
}}}}
\put(1126,-2311){\makebox(0,0)[lb]{\smash{{\SetFigFont{12}{14.4}{\rmdefault}{\mddefault}{\updefault}{\color[rgb]{0,0,0}$v_2$}%
}}}}
\put(5251,-2986){\makebox(0,0)[lb]{\smash{{\SetFigFont{12}{14.4}{\rmdefault}{\mddefault}{\updefault}{\color[rgb]{0,0,0}$a$}%
}}}}
\put(6076,-2986){\makebox(0,0)[lb]{\smash{{\SetFigFont{12}{14.4}{\rmdefault}{\mddefault}{\updefault}{\color[rgb]{0,0,0}$b$}%
}}}}
\put(9676,-1861){\makebox(0,0)[lb]{\smash{{\SetFigFont{12}{14.4}{\rmdefault}{\mddefault}{\updefault}{\color[rgb]{0,0,0}$a$}%
}}}}
\put(9676,-2986){\makebox(0,0)[lb]{\smash{{\SetFigFont{12}{14.4}{\rmdefault}{\mddefault}{\updefault}{\color[rgb]{0,0,0}$b$}%
}}}}
\put(8701,-2611){\makebox(0,0)[lb]{\smash{{\SetFigFont{12}{14.4}{\rmdefault}{\mddefault}{\updefault}{\color[rgb]{0,0,0}$c$}%
}}}}
\put(10276,-2611){\makebox(0,0)[lb]{\smash{{\SetFigFont{12}{14.4}{\rmdefault}{\mddefault}{\updefault}{\color[rgb]{0,0,0}$d$}%
}}}}
\put(2401,-1711){\makebox(0,0)[lb]{\smash{{\SetFigFont{12}{14.4}{\rmdefault}{\mddefault}{\updefault}{\color[rgb]{0,0,0}$v_3$}%
}}}}
\put(6076,-1786){\makebox(0,0)[lb]{\smash{{\SetFigFont{12}{14.4}{\rmdefault}{\mddefault}{\updefault}{\color[rgb]{0,0,0}$c$}%
}}}}
\put(3076,-2836){\makebox(0,0)[lb]{\smash{{\SetFigFont{12}{14.4}{\rmdefault}{\mddefault}{\updefault}{\color[rgb]{0,0,0}$C_{1}$}%
}}}}
\put(6676,-2836){\makebox(0,0)[lb]{\smash{{\SetFigFont{12}{14.4}{\rmdefault}{\mddefault}{\updefault}{\color[rgb]{0,0,0}$C_{1}$}%
}}}}
\put(9676,-2461){\makebox(0,0)[lb]{\smash{{\SetFigFont{12}{14.4}{\rmdefault}{\mddefault}{\updefault}{\color[rgb]{0,0,0}$C_{1}$}%
}}}}
\end{picture}%

%% file: egsm2.pstex_t
\begin{picture}(0,0)%
\includegraphics{egsm2.pstex}%
\end{picture}%
\setlength{\unitlength}{3947sp}%
\begingroup\makeatletter\ifx\SetFigFont\undefined%
\gdef\SetFigFont#1#2#3#4#5{%
  \reset@font\fontsize{#1}{#2pt}%
  \fontfamily{#3}\fontseries{#4}\fontshape{#5}%
  \selectfont}%
\fi\endgroup%
\begin{picture}(7005,2646)(2011,-3973)
\put(2326,-2986){\makebox(0,0)[lb]{\smash{{\SetFigFont{12}{14.4}{\rmdefault}{\mddefault}{\updefault}{\color[rgb]{0,0,0}$V_1$}%
}}}}
\put(3526,-2986){\makebox(0,0)[lb]{\smash{{\SetFigFont{12}{14.4}{\rmdefault}{\mddefault}{\updefault}{\color[rgb]{0,0,0}$V_2$}%
}}}}
\put(4426,-2386){\makebox(0,0)[lb]{\smash{{\SetFigFont{12}{14.4}{\rmdefault}{\mddefault}{\updefault}{\color[rgb]{0,0,0}$V_3$}%
}}}}
\put(2026,-1561){\makebox(0,0)[lb]{\smash{{\SetFigFont{12}{14.4}{\rmdefault}{\mddefault}{\updefault}{\color[rgb]{0,0,0}$V_m$}%
}}}}
\put(2926,-2986){\makebox(0,0)[lb]{\smash{{\SetFigFont{12}{14.4}{\rmdefault}{\mddefault}{\updefault}{\color[rgb]{0,0,0}$y_1$}%
}}}}
\put(2926,-2086){\makebox(0,0)[lb]{\smash{{\SetFigFont{12}{14.4}{\rmdefault}{\mddefault}{\updefault}{\color[rgb]{0,0,0}$y_2$}%
}}}}
\put(7651,-1486){\makebox(0,0)[lb]{\smash{{\SetFigFont{12}{14.4}{\rmdefault}{\mddefault}{\updefault}{\color[rgb]{0,0,0}$V_m$}%
}}}}
\put(9001,-2911){\makebox(0,0)[lb]{\smash{{\SetFigFont{12}{14.4}{\rmdefault}{\mddefault}{\updefault}{\color[rgb]{0,0,0}$V_2$}%
}}}}
\put(8401,-3511){\makebox(0,0)[lb]{\smash{{\SetFigFont{12}{14.4}{\rmdefault}{\mddefault}{\updefault}{\color[rgb]{0,0,0}$x_{1_2}$}%
}}}}
\put(8101,-2986){\makebox(0,0)[lb]{\smash{{\SetFigFont{12}{14.4}{\rmdefault}{\mddefault}{\updefault}{\color[rgb]{0,0,0}$\sigma_1$}%
}}}}
\put(2701,-2686){\makebox(0,0)[lb]{\smash{{\SetFigFont{12}{14.4}{\rmdefault}{\mddefault}{\updefault}{\color[rgb]{0,0,0}$\tau_1$}%
}}}}
\put(8401,-2086){\makebox(0,0)[lb]{\smash{{\SetFigFont{12}{14.4}{\rmdefault}{\mddefault}{\updefault}{\color[rgb]{0,0,0}$x_{1_1}$}%
}}}}
\end{picture}%

%% file: egch204.pstex_t
\begin{picture}(0,0)%
\includegraphics{egch204.pstex}%
\end{picture}%
\setlength{\unitlength}{3947sp}%
\begingroup\makeatletter\ifx\SetFigFont\undefined%
\gdef\SetFigFont#1#2#3#4#5{%
  \reset@font\fontsize{#1}{#2pt}%
  \fontfamily{#3}\fontseries{#4}\fontshape{#5}%
  \selectfont}%
\fi\endgroup%
\begin{picture}(9555,3511)(-1664,-5225)
\put(6451,-5161){\makebox(0,0)[lb]{\smash{{\SetFigFont{12}{14.4}{\rmdefault}{\mddefault}{\updefault}{\color[rgb]{0,0,0}$C$}%
}}}}
\put(5026,-4261){\makebox(0,0)[lb]{\smash{{\SetFigFont{12}{14.4}{\rmdefault}{\mddefault}{\updefault}{\color[rgb]{0,0,0}$O_{11}$}%
}}}}
\put(6601,-4186){\makebox(0,0)[lb]{\smash{{\SetFigFont{12}{14.4}{\rmdefault}{\mddefault}{\updefault}{\color[rgb]{0,0,0}$O_{01}$}%
}}}}
\put(7801,-4261){\makebox(0,0)[lb]{\smash{{\SetFigFont{12}{14.4}{\rmdefault}{\mddefault}{\updefault}{\color[rgb]{0,0,0}$O_{10}$}%
}}}}
\put(7351,-3436){\makebox(0,0)[lb]{\smash{{\SetFigFont{12}{14.4}{\rmdefault}{\mddefault}{\updefault}{\color[rgb]{0,0,0}$A$}%
}}}}
\put(6151,-3436){\makebox(0,0)[lb]{\smash{{\SetFigFont{12}{14.4}{\rmdefault}{\mddefault}{\updefault}{\color[rgb]{0,0,0}$B$}%
}}}}
\put(7876,-2761){\makebox(0,0)[lb]{\smash{{\SetFigFont{12}{14.4}{\rmdefault}{\mddefault}{\updefault}{\color[rgb]{0,0,0}$O_{11}$}%
}}}}
\put(6601,-2686){\makebox(0,0)[lb]{\smash{{\SetFigFont{12}{14.4}{\rmdefault}{\mddefault}{\updefault}{\color[rgb]{0,0,0}$O_{oo}$}%
}}}}
\put(5026,-2686){\makebox(0,0)[lb]{\smash{{\SetFigFont{12}{14.4}{\rmdefault}{\mddefault}{\updefault}{\color[rgb]{0,0,0}$O_{10}$}%
}}}}
\put(6451,-1861){\makebox(0,0)[lb]{\smash{{\SetFigFont{12}{14.4}{\rmdefault}{\mddefault}{\updefault}{\color[rgb]{0,0,0}$C$}%
}}}}
\put(3901,-3286){\makebox(0,0)[lb]{\smash{{\SetFigFont{12}{14.4}{\rmdefault}{\mddefault}{\updefault}{\color[rgb]{0,0,0}$1$}%
}}}}
\put(2251,-3436){\makebox(0,0)[lb]{\smash{{\SetFigFont{12}{14.4}{\rmdefault}{\mddefault}{\updefault}{\color[rgb]{0,0,0}$2$}%
}}}}
\put(2926,-4486){\makebox(0,0)[lb]{\smash{{\SetFigFont{12}{14.4}{\rmdefault}{\mddefault}{\updefault}{\color[rgb]{0,0,0}$3$}%
}}}}
\put(2851,-2386){\makebox(0,0)[lb]{\smash{{\SetFigFont{12}{14.4}{\rmdefault}{\mddefault}{\updefault}{\color[rgb]{0,0,0}$4$}%
}}}}
\put(1876,-2461){\makebox(0,0)[lb]{\smash{{\SetFigFont{12}{14.4}{\rmdefault}{\mddefault}{\updefault}{\color[rgb]{0,0,0}$6$}%
}}}}
\put(3826,-4486){\makebox(0,0)[lb]{\smash{{\SetFigFont{12}{14.4}{\rmdefault}{\mddefault}{\updefault}{\color[rgb]{0,0,0}$6$}%
}}}}
\put(1876,-4486){\makebox(0,0)[lb]{\smash{{\SetFigFont{12}{14.4}{\rmdefault}{\mddefault}{\updefault}{\color[rgb]{0,0,0}$5$}%
}}}}
\put(3901,-2461){\makebox(0,0)[lb]{\smash{{\SetFigFont{12}{14.4}{\rmdefault}{\mddefault}{\updefault}{\color[rgb]{0,0,0}$5$}%
}}}}
\put(301,-2611){\makebox(0,0)[lb]{\smash{{\SetFigFont{12}{14.4}{\rmdefault}{\mddefault}{\updefault}{\color[rgb]{0,0,0}$A$}%
}}}}
\put(376,-3886){\makebox(0,0)[lb]{\smash{{\SetFigFont{12}{14.4}{\rmdefault}{\mddefault}{\updefault}{\color[rgb]{0,0,0}$D$}%
}}}}
\put(-824,-3886){\makebox(0,0)[lb]{\smash{{\SetFigFont{12}{14.4}{\rmdefault}{\mddefault}{\updefault}{\color[rgb]{0,0,0}$C$}%
}}}}
\put(-824,-2761){\makebox(0,0)[lb]{\smash{{\SetFigFont{12}{14.4}{\rmdefault}{\mddefault}{\updefault}{\color[rgb]{0,0,0}$B$}%
}}}}
\put(-449,-5086){\makebox(0,0)[lb]{\smash{{\SetFigFont{12}{14.4}{\rmdefault}{\mddefault}{\updefault}{\color[rgb]{0,0,0}$(a)$}%
}}}}
\put(5626,-5086){\makebox(0,0)[lb]{\smash{{\SetFigFont{12}{14.4}{\rmdefault}{\mddefault}{\updefault}{\color[rgb]{0,0,0}$(c)$}%
}}}}
\put(2851,-5011){\makebox(0,0)[lb]{\smash{{\SetFigFont{12}{14.4}{\rmdefault}{\mddefault}{\updefault}{\color[rgb]{0,0,0}$(b)$}%
}}}}
\put(-1499,-4786){\makebox(0,0)[lb]{\smash{{\SetFigFont{12}{14.4}{\rmdefault}{\mddefault}{\updefault}{\color[rgb]{0,0,0}1}%
}}}}
\put(-974,-4786){\makebox(0,0)[lb]{\smash{{\SetFigFont{12}{14.4}{\rmdefault}{\mddefault}{\updefault}{\color[rgb]{0,0,0}2}%
}}}}
\put(-299,-4786){\makebox(0,0)[lb]{\smash{{\SetFigFont{12}{14.4}{\rmdefault}{\mddefault}{\updefault}{\color[rgb]{0,0,0}3}%
}}}}
\put(301,-4786){\makebox(0,0)[lb]{\smash{{\SetFigFont{12}{14.4}{\rmdefault}{\mddefault}{\updefault}{\color[rgb]{0,0,0}4}%
}}}}
\put(901,-4786){\makebox(0,0)[lb]{\smash{{\SetFigFont{12}{14.4}{\rmdefault}{\mddefault}{\updefault}{\color[rgb]{0,0,0}1}%
}}}}
\put(-1649,-4036){\makebox(0,0)[lb]{\smash{{\SetFigFont{12}{14.4}{\rmdefault}{\mddefault}{\updefault}{\color[rgb]{0,0,0}5}%
}}}}
\put(-1649,-3511){\makebox(0,0)[lb]{\smash{{\SetFigFont{12}{14.4}{\rmdefault}{\mddefault}{\updefault}{\color[rgb]{0,0,0}6}%
}}}}
\put(-1649,-2836){\makebox(0,0)[lb]{\smash{{\SetFigFont{12}{14.4}{\rmdefault}{\mddefault}{\updefault}{\color[rgb]{0,0,0}7}%
}}}}
\put(-1649,-2161){\makebox(0,0)[lb]{\smash{{\SetFigFont{12}{14.4}{\rmdefault}{\mddefault}{\updefault}{\color[rgb]{0,0,0}1}%
}}}}
\put(-974,-2086){\makebox(0,0)[lb]{\smash{{\SetFigFont{12}{14.4}{\rmdefault}{\mddefault}{\updefault}{\color[rgb]{0,0,0}2}%
}}}}
\put(-374,-2086){\makebox(0,0)[lb]{\smash{{\SetFigFont{12}{14.4}{\rmdefault}{\mddefault}{\updefault}{\color[rgb]{0,0,0}3}%
}}}}
\put(226,-2086){\makebox(0,0)[lb]{\smash{{\SetFigFont{12}{14.4}{\rmdefault}{\mddefault}{\updefault}{\color[rgb]{0,0,0}4}%
}}}}
\put(901,-2086){\makebox(0,0)[lb]{\smash{{\SetFigFont{12}{14.4}{\rmdefault}{\mddefault}{\updefault}{\color[rgb]{0,0,0}1}%
}}}}
\put(976,-4036){\makebox(0,0)[lb]{\smash{{\SetFigFont{12}{14.4}{\rmdefault}{\mddefault}{\updefault}{\color[rgb]{0,0,0}5}%
}}}}
\put(976,-3511){\makebox(0,0)[lb]{\smash{{\SetFigFont{12}{14.4}{\rmdefault}{\mddefault}{\updefault}{\color[rgb]{0,0,0}6}%
}}}}
\put(976,-2836){\makebox(0,0)[lb]{\smash{{\SetFigFont{12}{14.4}{\rmdefault}{\mddefault}{\updefault}{\color[rgb]{0,0,0}7}%
}}}}
\end{picture}%

%% file: egch205.pstex_t
\begin{picture}(0,0)%
\includegraphics{egch205.pstex}%
\end{picture}%
\setlength{\unitlength}{4144sp}%
\begingroup\makeatletter\ifx\SetFigFont\undefined%
\gdef\SetFigFont#1#2#3#4#5{%
  \reset@font\fontsize{#1}{#2pt}%
  \fontfamily{#3}\fontseries{#4}\fontshape{#5}%
  \selectfont}%
\fi\endgroup%
\begin{picture}(10958,3456)(-1544,-3505)
\put(2071,-2086){\makebox(0,0)[lb]{\smash{{\SetFigFont{12}{14.4}{\rmdefault}{\mddefault}{\updefault}{\color[rgb]{0,0,0}$1$}%
}}}}
\put(2656,-2941){\makebox(0,0)[lb]{\smash{{\SetFigFont{12}{14.4}{\rmdefault}{\mddefault}{\updefault}{\color[rgb]{0,0,0}$2$}%
}}}}
\put(1216,-2986){\makebox(0,0)[lb]{\smash{{\SetFigFont{12}{14.4}{\rmdefault}{\mddefault}{\updefault}{\color[rgb]{0,0,0}$3$}%
}}}}
\put(2026,-241){\makebox(0,0)[lb]{\smash{{\SetFigFont{12}{14.4}{\rmdefault}{\mddefault}{\updefault}{\color[rgb]{0,0,0}$4$}%
}}}}
\put(2746,-1006){\makebox(0,0)[lb]{\smash{{\SetFigFont{12}{14.4}{\rmdefault}{\mddefault}{\updefault}{\color[rgb]{0,0,0}$5$}%
}}}}
\put(1036,-916){\makebox(0,0)[lb]{\smash{{\SetFigFont{12}{14.4}{\rmdefault}{\mddefault}{\updefault}{\color[rgb]{0,0,0}$6$}%
}}}}
\put(5221,-2086){\makebox(0,0)[lb]{\smash{{\SetFigFont{12}{14.4}{\rmdefault}{\mddefault}{\updefault}{\color[rgb]{0,0,0}$1$}%
}}}}
\put(5806,-2986){\makebox(0,0)[lb]{\smash{{\SetFigFont{12}{14.4}{\rmdefault}{\mddefault}{\updefault}{\color[rgb]{0,0,0}$2$}%
}}}}
\put(4366,-2986){\makebox(0,0)[lb]{\smash{{\SetFigFont{12}{14.4}{\rmdefault}{\mddefault}{\updefault}{\color[rgb]{0,0,0}$3$}%
}}}}
\put(5131,-196){\makebox(0,0)[lb]{\smash{{\SetFigFont{12}{14.4}{\rmdefault}{\mddefault}{\updefault}{\color[rgb]{0,0,0}$4$}%
}}}}
\put(5896,-1006){\makebox(0,0)[lb]{\smash{{\SetFigFont{12}{14.4}{\rmdefault}{\mddefault}{\updefault}{\color[rgb]{0,0,0}$5$}%
}}}}
\put(4231,-871){\makebox(0,0)[lb]{\smash{{\SetFigFont{12}{14.4}{\rmdefault}{\mddefault}{\updefault}{\color[rgb]{0,0,0}$6$}%
}}}}
\put(8371,-2086){\makebox(0,0)[lb]{\smash{{\SetFigFont{12}{14.4}{\rmdefault}{\mddefault}{\updefault}{\color[rgb]{0,0,0}$1$}%
}}}}
\put(9046,-2896){\makebox(0,0)[lb]{\smash{{\SetFigFont{12}{14.4}{\rmdefault}{\mddefault}{\updefault}{\color[rgb]{0,0,0}$2$}%
}}}}
\put(7516,-2941){\makebox(0,0)[lb]{\smash{{\SetFigFont{12}{14.4}{\rmdefault}{\mddefault}{\updefault}{\color[rgb]{0,0,0}$3$}%
}}}}
\put(8281,-241){\makebox(0,0)[lb]{\smash{{\SetFigFont{12}{14.4}{\rmdefault}{\mddefault}{\updefault}{\color[rgb]{0,0,0}$4$}%
}}}}
\put(9046,-1006){\makebox(0,0)[lb]{\smash{{\SetFigFont{12}{14.4}{\rmdefault}{\mddefault}{\updefault}{\color[rgb]{0,0,0}$5$}%
}}}}
\put(7381,-961){\makebox(0,0)[lb]{\smash{{\SetFigFont{12}{14.4}{\rmdefault}{\mddefault}{\updefault}{\color[rgb]{0,0,0}$6$}%
}}}}
\put(1981,-3301){\makebox(0,0)[lb]{\smash{{\SetFigFont{12}{14.4}{\rmdefault}{\mddefault}{\updefault}{\color[rgb]{0,0,0}$(1, 0, 0)$}%
}}}}
\put(5041,-3256){\makebox(0,0)[lb]{\smash{{\SetFigFont{12}{14.4}{\rmdefault}{\mddefault}{\updefault}{\color[rgb]{0,0,0}$(1, 0, 0)$}%
}}}}
\put(8281,-3211){\makebox(0,0)[lb]{\smash{{\SetFigFont{12}{14.4}{\rmdefault}{\mddefault}{\updefault}{\color[rgb]{0,0,0}$(1, 0, 0)$}%
}}}}
\put(541,-2311){\makebox(0,0)[lb]{\smash{{\SetFigFont{12}{14.4}{\rmdefault}{\mddefault}{\updefault}{\color[rgb]{0,0,0}$(1, 1, 1)$}%
}}}}
\put(631,-1771){\makebox(0,0)[lb]{\smash{{\SetFigFont{12}{14.4}{\rmdefault}{\mddefault}{\updefault}{\color[rgb]{0,0,0}$(0,1,0)$}%
}}}}
\put(2746,-1816){\makebox(0,0)[lb]{\smash{{\SetFigFont{12}{14.4}{\rmdefault}{\mddefault}{\updefault}{\color[rgb]{0,0,0}$(0,0,1)$}%
}}}}
\put(2791,-511){\makebox(0,0)[lb]{\smash{{\SetFigFont{12}{14.4}{\rmdefault}{\mddefault}{\updefault}{\color[rgb]{0,0,0}$(1,0,0)$}%
}}}}
\put(3691,-1366){\makebox(0,0)[lb]{\smash{{\SetFigFont{12}{14.4}{\rmdefault}{\mddefault}{\updefault}{\color[rgb]{0,0,0}$(0,1,1)$}%
}}}}
\put(5941,-1591){\makebox(0,0)[lb]{\smash{{\SetFigFont{12}{14.4}{\rmdefault}{\mddefault}{\updefault}{\color[rgb]{0,0,0}$(0,0,1)$}%
}}}}
\put(5806,-511){\makebox(0,0)[lb]{\smash{{\SetFigFont{12}{14.4}{\rmdefault}{\mddefault}{\updefault}{\color[rgb]{0,0,0}$(1,1,1)$}%
}}}}
\put(3736,-2221){\makebox(0,0)[lb]{\smash{{\SetFigFont{12}{14.4}{\rmdefault}{\mddefault}{\updefault}{\color[rgb]{0,0,0}$(0,1,0)$}%
}}}}
\put(9136,-1996){\makebox(0,0)[lb]{\smash{{\SetFigFont{12}{14.4}{\rmdefault}{\mddefault}{\updefault}{\color[rgb]{0,0,0}$(0,1,1)$}%
}}}}
\put(9046,-1501){\makebox(0,0)[lb]{\smash{{\SetFigFont{12}{14.4}{\rmdefault}{\mddefault}{\updefault}{\color[rgb]{0,0,0}$(0,0,1)$}%
}}}}
\put(6931,-2266){\makebox(0,0)[lb]{\smash{{\SetFigFont{12}{14.4}{\rmdefault}{\mddefault}{\updefault}{\color[rgb]{0,0,0}$(0,1,0)$}%
}}}}
\put(8956,-466){\makebox(0,0)[lb]{\smash{{\SetFigFont{12}{14.4}{\rmdefault}{\mddefault}{\updefault}{\color[rgb]{0,0,0}$(1, 0,0)$}%
}}}}
\put(-584,-3256){\makebox(0,0)[lb]{\smash{{\SetFigFont{12}{14.4}{\rmdefault}{\mddefault}{\updefault}{\color[rgb]{0,0,0}$(1, 0, 0)$}%
}}}}
\put(-1349,-3391){\makebox(0,0)[lb]{\smash{{\SetFigFont{12}{14.4}{\rmdefault}{\mddefault}{\updefault}{\color[rgb]{0,0,0}$(0, 1, 0)$}%
}}}}
\put(-224,-1096){\makebox(0,0)[lb]{\smash{{\SetFigFont{12}{14.4}{\rmdefault}{\mddefault}{\updefault}{\color[rgb]{0,0,0}$(1, 1, 1)$}%
}}}}
\put(-944,-736){\makebox(0,0)[lb]{\smash{{\SetFigFont{12}{14.4}{\rmdefault}{\mddefault}{\updefault}{\color[rgb]{0,0,0}$(0, 0, 1)$}%
}}}}
\put(226,-3436){\makebox(0,0)[lb]{\smash{{\SetFigFont{12}{14.4}{\rmdefault}{\mddefault}{\updefault}{\color[rgb]{0,0,0}$(a)$}%
}}}}
\put(1351,-3346){\makebox(0,0)[lb]{\smash{{\SetFigFont{12}{14.4}{\rmdefault}{\mddefault}{\updefault}{\color[rgb]{0,0,0}$(b)$}%
}}}}
\put(4546,-3301){\makebox(0,0)[lb]{\smash{{\SetFigFont{12}{14.4}{\rmdefault}{\mddefault}{\updefault}{\color[rgb]{0,0,0}$(c)$}%
}}}}
\put(7696,-3256){\makebox(0,0)[lb]{\smash{{\SetFigFont{12}{14.4}{\rmdefault}{\mddefault}{\updefault}{\color[rgb]{0,0,0}$(d)$}%
}}}}
\put(-1529,-916){\makebox(0,0)[lb]{\smash{{\SetFigFont{12}{14.4}{\rmdefault}{\mddefault}{\updefault}{\color[rgb]{0,0,0}$1$}%
}}}}
\put(-1394,-2986){\makebox(0,0)[lb]{\smash{{\SetFigFont{12}{14.4}{\rmdefault}{\mddefault}{\updefault}{\color[rgb]{0,0,0}$2$}%
}}}}
\put( 46,-2896){\makebox(0,0)[lb]{\smash{{\SetFigFont{12}{14.4}{\rmdefault}{\mddefault}{\updefault}{\color[rgb]{0,0,0}$3$}%
}}}}
\put(-494,-1546){\makebox(0,0)[lb]{\smash{{\SetFigFont{12}{14.4}{\rmdefault}{\mddefault}{\updefault}{\color[rgb]{0,0,0}$4$}%
}}}}
\end{picture}%

%% file: egch203.pstex_t
\begin{picture}(0,0)%
\includegraphics{egch203.pstex}%
\end{picture}%
\setlength{\unitlength}{3947sp}%
\begingroup\makeatletter\ifx\SetFigFont\undefined%
\gdef\SetFigFont#1#2#3#4#5{%
  \reset@font\fontsize{#1}{#2pt}%
  \fontfamily{#3}\fontseries{#4}\fontshape{#5}%
  \selectfont}%
\fi\endgroup%
\begin{picture}(14130,2986)(886,-2525)
\put(2701,-1111){\makebox(0,0)[lb]{\smash{{\SetFigFont{12}{14.4}{\rmdefault}{\mddefault}{\updefault}{\color[rgb]{0,0,0}$1$}%
}}}}
\put(6226,-811){\makebox(0,0)[lb]{\smash{{\SetFigFont{12}{14.4}{\rmdefault}{\mddefault}{\updefault}{\color[rgb]{0,0,0}$2$}%
}}}}
\put(10351,-1036){\makebox(0,0)[lb]{\smash{{\SetFigFont{12}{14.4}{\rmdefault}{\mddefault}{\updefault}{\color[rgb]{0,0,0}$3$}%
}}}}
\put(13951,-1036){\makebox(0,0)[lb]{\smash{{\SetFigFont{12}{14.4}{\rmdefault}{\mddefault}{\updefault}{\color[rgb]{0,0,0}$4$}%
}}}}
\put(1051,-2386){\makebox(0,0)[lb]{\smash{{\SetFigFont{12}{14.4}{\rmdefault}{\mddefault}{\updefault}{\color[rgb]{0,0,0}$a$}%
}}}}
\put(2776,-2461){\makebox(0,0)[lb]{\smash{{\SetFigFont{12}{14.4}{\rmdefault}{\mddefault}{\updefault}{\color[rgb]{0,0,0}$b$}%
}}}}
\put(3976,164){\makebox(0,0)[lb]{\smash{{\SetFigFont{12}{14.4}{\rmdefault}{\mddefault}{\updefault}{\color[rgb]{0,0,0}$c$}%
}}}}
\put(901,-361){\makebox(0,0)[lb]{\smash{{\SetFigFont{12}{14.4}{\rmdefault}{\mddefault}{\updefault}{\color[rgb]{0,0,0}$d$}%
}}}}
\put(1876,314){\makebox(0,0)[lb]{\smash{{\SetFigFont{12}{14.4}{\rmdefault}{\mddefault}{\updefault}{\color[rgb]{0,0,0}$e$}%
}}}}
\put(2776,-286){\makebox(0,0)[lb]{\smash{{\SetFigFont{12}{14.4}{\rmdefault}{\mddefault}{\updefault}{\color[rgb]{0,0,0}$f$}%
}}}}
\put(3901,-1561){\makebox(0,0)[lb]{\smash{{\SetFigFont{12}{14.4}{\rmdefault}{\mddefault}{\updefault}{\color[rgb]{0,0,0}$g$}%
}}}}
\put(2101,-1486){\makebox(0,0)[lb]{\smash{{\SetFigFont{12}{14.4}{\rmdefault}{\mddefault}{\updefault}{\color[rgb]{0,0,0}$h$}%
}}}}
\put(5926,-1786){\makebox(0,0)[lb]{\smash{{\SetFigFont{12}{14.4}{\rmdefault}{\mddefault}{\updefault}{\color[rgb]{0,0,0}$a$}%
}}}}
\put(7876,-1636){\makebox(0,0)[lb]{\smash{{\SetFigFont{12}{14.4}{\rmdefault}{\mddefault}{\updefault}{\color[rgb]{0,0,0}$b$}%
}}}}
\put(4951,-2461){\makebox(0,0)[lb]{\smash{{\SetFigFont{12}{14.4}{\rmdefault}{\mddefault}{\updefault}{\color[rgb]{0,0,0}$c$}%
}}}}
\put(5926,314){\makebox(0,0)[lb]{\smash{{\SetFigFont{12}{14.4}{\rmdefault}{\mddefault}{\updefault}{\color[rgb]{0,0,0}$d$}%
}}}}
\put(6826,-2386){\makebox(0,0)[lb]{\smash{{\SetFigFont{12}{14.4}{\rmdefault}{\mddefault}{\updefault}{\color[rgb]{0,0,0}$e$}%
}}}}
\put(7876,164){\makebox(0,0)[lb]{\smash{{\SetFigFont{12}{14.4}{\rmdefault}{\mddefault}{\updefault}{\color[rgb]{0,0,0}$f$}%
}}}}
\put(4876,-286){\makebox(0,0)[lb]{\smash{{\SetFigFont{12}{14.4}{\rmdefault}{\mddefault}{\updefault}{\color[rgb]{0,0,0}$g$}%
}}}}
\put(6976,-511){\makebox(0,0)[lb]{\smash{{\SetFigFont{12}{14.4}{\rmdefault}{\mddefault}{\updefault}{\color[rgb]{0,0,0}$h$}%
}}}}
\put(9526,314){\makebox(0,0)[lb]{\smash{{\SetFigFont{12}{14.4}{\rmdefault}{\mddefault}{\updefault}{\color[rgb]{0,0,0}$a$}%
}}}}
\put(11401,239){\makebox(0,0)[lb]{\smash{{\SetFigFont{12}{14.4}{\rmdefault}{\mddefault}{\updefault}{\color[rgb]{0,0,0}$b$}%
}}}}
\put(8776,-361){\makebox(0,0)[lb]{\smash{{\SetFigFont{12}{14.4}{\rmdefault}{\mddefault}{\updefault}{\color[rgb]{0,0,0}$c$}%
}}}}
\put(10801,-2386){\makebox(0,0)[lb]{\smash{{\SetFigFont{12}{14.4}{\rmdefault}{\mddefault}{\updefault}{\color[rgb]{0,0,0}$d$}%
}}}}
\put(10876,-511){\makebox(0,0)[lb]{\smash{{\SetFigFont{12}{14.4}{\rmdefault}{\mddefault}{\updefault}{\color[rgb]{0,0,0}$e$}%
}}}}
\put(8851,-2461){\makebox(0,0)[lb]{\smash{{\SetFigFont{12}{14.4}{\rmdefault}{\mddefault}{\updefault}{\color[rgb]{0,0,0}$f$}%
}}}}
\put(11476,-1711){\makebox(0,0)[lb]{\smash{{\SetFigFont{12}{14.4}{\rmdefault}{\mddefault}{\updefault}{\color[rgb]{0,0,0}$g$}%
}}}}
\put(9301,-1636){\makebox(0,0)[lb]{\smash{{\SetFigFont{12}{14.4}{\rmdefault}{\mddefault}{\updefault}{\color[rgb]{0,0,0}$h$}%
}}}}
\put(15001,-1786){\makebox(0,0)[lb]{\smash{{\SetFigFont{12}{14.4}{\rmdefault}{\mddefault}{\updefault}{\color[rgb]{0,0,0}$a$}%
}}}}
\put(12376,-361){\makebox(0,0)[lb]{\smash{{\SetFigFont{12}{14.4}{\rmdefault}{\mddefault}{\updefault}{\color[rgb]{0,0,0}$b$}%
}}}}
\put(14401,-2386){\makebox(0,0)[lb]{\smash{{\SetFigFont{12}{14.4}{\rmdefault}{\mddefault}{\updefault}{\color[rgb]{0,0,0}$c$}%
}}}}
\put(14926,314){\makebox(0,0)[lb]{\smash{{\SetFigFont{12}{14.4}{\rmdefault}{\mddefault}{\updefault}{\color[rgb]{0,0,0}$d$}%
}}}}
\put(13051,314){\makebox(0,0)[lb]{\smash{{\SetFigFont{12}{14.4}{\rmdefault}{\mddefault}{\updefault}{\color[rgb]{0,0,0}$e$}%
}}}}
\put(12451,-2386){\makebox(0,0)[lb]{\smash{{\SetFigFont{12}{14.4}{\rmdefault}{\mddefault}{\updefault}{\color[rgb]{0,0,0}$f$}%
}}}}
\put(14476,-511){\makebox(0,0)[lb]{\smash{{\SetFigFont{12}{14.4}{\rmdefault}{\mddefault}{\updefault}{\color[rgb]{0,0,0}$g$}%
}}}}
\put(12976,-1561){\makebox(0,0)[lb]{\smash{{\SetFigFont{12}{14.4}{\rmdefault}{\mddefault}{\updefault}{\color[rgb]{0,0,0}$h$}%
}}}}
\end{picture}%

%% file: egch201c.pstex_t
\begin{picture}(0,0)%
\includegraphics{egch201c.pstex}%
\end{picture}%
\setlength{\unitlength}{3947sp}%
\begingroup\makeatletter\ifx\SetFigFont\undefined%
\gdef\SetFigFont#1#2#3#4#5{%
  \reset@font\fontsize{#1}{#2pt}%
  \fontfamily{#3}\fontseries{#4}\fontshape{#5}%
  \selectfont}%
\fi\endgroup%
\begin{picture}(9405,6291)(886,-11230)
\put(1951,-7936){\makebox(0,0)[lb]{\smash{{\SetFigFont{12}{14.4}{\rmdefault}{\mddefault}{\updefault}{\color[rgb]{0,0,0}$Z_1$}%
}}}}
\put(5176,-7936){\makebox(0,0)[lb]{\smash{{\SetFigFont{12}{14.4}{\rmdefault}{\mddefault}{\updefault}{\color[rgb]{0,0,0}$Z_2$}%
}}}}
\put(8551,-8011){\makebox(0,0)[lb]{\smash{{\SetFigFont{12}{14.4}{\rmdefault}{\mddefault}{\updefault}{\color[rgb]{0,0,0}$Z_3$}%
}}}}
\put(2776,-5986){\makebox(0,0)[lb]{\smash{{\SetFigFont{12}{14.4}{\rmdefault}{\mddefault}{\updefault}{\color[rgb]{0,0,0}$d$}%
}}}}
\put(7576,-9211){\makebox(0,0)[lb]{\smash{{\SetFigFont{12}{14.4}{\rmdefault}{\mddefault}{\updefault}{\color[rgb]{0,0,0}$d$}%
}}}}
\put(3676,-5161){\makebox(0,0)[lb]{\smash{{\SetFigFont{12}{14.4}{\rmdefault}{\mddefault}{\updefault}{\color[rgb]{0,0,0}$a$}%
}}}}
\put(10201,-5161){\makebox(0,0)[lb]{\smash{{\SetFigFont{12}{14.4}{\rmdefault}{\mddefault}{\updefault}{\color[rgb]{0,0,0}$a$}%
}}}}
\put(1801,-5086){\makebox(0,0)[lb]{\smash{{\SetFigFont{12}{14.4}{\rmdefault}{\mddefault}{\updefault}{\color[rgb]{0,0,0}$c$}%
}}}}
\put(901,-5686){\makebox(0,0)[lb]{\smash{{\SetFigFont{12}{14.4}{\rmdefault}{\mddefault}{\updefault}{\color[rgb]{0,0,0}$g$}%
}}}}
\put(2851,-7786){\makebox(0,0)[lb]{\smash{{\SetFigFont{12}{14.4}{\rmdefault}{\mddefault}{\updefault}{\color[rgb]{0,0,0}$f$}%
}}}}
\put(3676,-7036){\makebox(0,0)[lb]{\smash{{\SetFigFont{12}{14.4}{\rmdefault}{\mddefault}{\updefault}{\color[rgb]{0,0,0}$b$}%
}}}}
\put(10276,-7036){\makebox(0,0)[lb]{\smash{{\SetFigFont{12}{14.4}{\rmdefault}{\mddefault}{\updefault}{\color[rgb]{0,0,0}$b$}%
}}}}
\put(1501,-6961){\makebox(0,0)[lb]{\smash{{\SetFigFont{12}{14.4}{\rmdefault}{\mddefault}{\updefault}{\color[rgb]{0,0,0}$e$}%
}}}}
\put(1201,-7786){\makebox(0,0)[lb]{\smash{{\SetFigFont{12}{14.4}{\rmdefault}{\mddefault}{\updefault}{\color[rgb]{0,0,0}$h$}%
}}}}
\put(2251,-10786){\makebox(0,0)[lb]{\smash{{\SetFigFont{12}{14.4}{\rmdefault}{\mddefault}{\updefault}{\color[rgb]{0,0,0}$r$}%
}}}}
\put(6901,-5161){\makebox(0,0)[lb]{\smash{{\SetFigFont{12}{14.4}{\rmdefault}{\mddefault}{\updefault}{\color[rgb]{0,0,0}$a$}%
}}}}
\put(6976,-7111){\makebox(0,0)[lb]{\smash{{\SetFigFont{12}{14.4}{\rmdefault}{\mddefault}{\updefault}{\color[rgb]{0,0,0}$b$}%
}}}}
\put(6376,-5911){\makebox(0,0)[lb]{\smash{{\SetFigFont{12}{14.4}{\rmdefault}{\mddefault}{\updefault}{\color[rgb]{0,0,0}$d$}%
}}}}
\put(6226,-7786){\makebox(0,0)[lb]{\smash{{\SetFigFont{12}{14.4}{\rmdefault}{\mddefault}{\updefault}{\color[rgb]{0,0,0}$f$}%
}}}}
\put(5026,-5086){\makebox(0,0)[lb]{\smash{{\SetFigFont{12}{14.4}{\rmdefault}{\mddefault}{\updefault}{\color[rgb]{0,0,0}$h$}%
}}}}
\put(4801,-6961){\makebox(0,0)[lb]{\smash{{\SetFigFont{12}{14.4}{\rmdefault}{\mddefault}{\updefault}{\color[rgb]{0,0,0}$g$}%
}}}}
\put(4276,-5761){\makebox(0,0)[lb]{\smash{{\SetFigFont{12}{14.4}{\rmdefault}{\mddefault}{\updefault}{\color[rgb]{0,0,0}$e$}%
}}}}
\put(4426,-7786){\makebox(0,0)[lb]{\smash{{\SetFigFont{12}{14.4}{\rmdefault}{\mddefault}{\updefault}{\color[rgb]{0,0,0}$c$}%
}}}}
\put(9676,-5911){\makebox(0,0)[lb]{\smash{{\SetFigFont{12}{14.4}{\rmdefault}{\mddefault}{\updefault}{\color[rgb]{0,0,0}$h$}%
}}}}
\put(9526,-7786){\makebox(0,0)[lb]{\smash{{\SetFigFont{12}{14.4}{\rmdefault}{\mddefault}{\updefault}{\color[rgb]{0,0,0}$g$}%
}}}}
\put(8326,-5086){\makebox(0,0)[lb]{\smash{{\SetFigFont{12}{14.4}{\rmdefault}{\mddefault}{\updefault}{\color[rgb]{0,0,0}$c$}%
}}}}
\put(8101,-7036){\makebox(0,0)[lb]{\smash{{\SetFigFont{12}{14.4}{\rmdefault}{\mddefault}{\updefault}{\color[rgb]{0,0,0}$e$}%
}}}}
\put(7576,-5836){\makebox(0,0)[lb]{\smash{{\SetFigFont{12}{14.4}{\rmdefault}{\mddefault}{\updefault}{\color[rgb]{0,0,0}$f$}%
}}}}
\put(7726,-7786){\makebox(0,0)[lb]{\smash{{\SetFigFont{12}{14.4}{\rmdefault}{\mddefault}{\updefault}{\color[rgb]{0,0,0}$d$}%
}}}}
\put(3076,-10111){\makebox(0,0)[lb]{\smash{{\SetFigFont{12}{14.4}{\rmdefault}{\mddefault}{\updefault}{\color[rgb]{0,0,0}$p$}%
}}}}
\put(3676,-9436){\makebox(0,0)[lb]{\smash{{\SetFigFont{12}{14.4}{\rmdefault}{\mddefault}{\updefault}{\color[rgb]{0,0,0}$q$}%
}}}}
\put(1876,-9361){\makebox(0,0)[lb]{\smash{{\SetFigFont{12}{14.4}{\rmdefault}{\mddefault}{\updefault}{\color[rgb]{0,0,0}$r$}%
}}}}
\put(2326,-9511){\makebox(0,0)[lb]{\smash{{\SetFigFont{12}{14.4}{\rmdefault}{\mddefault}{\updefault}{\color[rgb]{0,0,0}$4$}%
}}}}
\put(976,-10036){\makebox(0,0)[lb]{\smash{{\SetFigFont{12}{14.4}{\rmdefault}{\mddefault}{\updefault}{\color[rgb]{0,0,0}$s$}%
}}}}
\put(3601,-8461){\makebox(0,0)[lb]{\smash{{\SetFigFont{12}{14.4}{\rmdefault}{\mddefault}{\updefault}{\color[rgb]{0,0,0}$b$}%
}}}}
\put(3676,-10336){\makebox(0,0)[lb]{\smash{{\SetFigFont{12}{14.4}{\rmdefault}{\mddefault}{\updefault}{\color[rgb]{0,0,0}$a$}%
}}}}
\put(3076,-9211){\makebox(0,0)[lb]{\smash{{\SetFigFont{12}{14.4}{\rmdefault}{\mddefault}{\updefault}{\color[rgb]{0,0,0}$f$}%
}}}}
\put(2926,-11086){\makebox(0,0)[lb]{\smash{{\SetFigFont{12}{14.4}{\rmdefault}{\mddefault}{\updefault}{\color[rgb]{0,0,0}$d$}%
}}}}
\put(1726,-8386){\makebox(0,0)[lb]{\smash{{\SetFigFont{12}{14.4}{\rmdefault}{\mddefault}{\updefault}{\color[rgb]{0,0,0}$e$}%
}}}}
\put(1876,-10186){\makebox(0,0)[lb]{\smash{{\SetFigFont{12}{14.4}{\rmdefault}{\mddefault}{\updefault}{\color[rgb]{0,0,0}$c$}%
}}}}
\put(976,-8986){\makebox(0,0)[lb]{\smash{{\SetFigFont{12}{14.4}{\rmdefault}{\mddefault}{\updefault}{\color[rgb]{0,0,0}$h$}%
}}}}
\put(1126,-11086){\makebox(0,0)[lb]{\smash{{\SetFigFont{12}{14.4}{\rmdefault}{\mddefault}{\updefault}{\color[rgb]{0,0,0}$g$}%
}}}}
\put(6901,-8461){\makebox(0,0)[lb]{\smash{{\SetFigFont{12}{14.4}{\rmdefault}{\mddefault}{\updefault}{\color[rgb]{0,0,0}$b$}%
}}}}
\put(6976,-10336){\makebox(0,0)[lb]{\smash{{\SetFigFont{12}{14.4}{\rmdefault}{\mddefault}{\updefault}{\color[rgb]{0,0,0}$a$}%
}}}}
\put(6376,-9211){\makebox(0,0)[lb]{\smash{{\SetFigFont{12}{14.4}{\rmdefault}{\mddefault}{\updefault}{\color[rgb]{0,0,0}$f$}%
}}}}
\put(6226,-11086){\makebox(0,0)[lb]{\smash{{\SetFigFont{12}{14.4}{\rmdefault}{\mddefault}{\updefault}{\color[rgb]{0,0,0}$d$}%
}}}}
\put(5101,-8386){\makebox(0,0)[lb]{\smash{{\SetFigFont{12}{14.4}{\rmdefault}{\mddefault}{\updefault}{\color[rgb]{0,0,0}$g$}%
}}}}
\put(4276,-9136){\makebox(0,0)[lb]{\smash{{\SetFigFont{12}{14.4}{\rmdefault}{\mddefault}{\updefault}{\color[rgb]{0,0,0}$c$}%
}}}}
\put(4426,-11086){\makebox(0,0)[lb]{\smash{{\SetFigFont{12}{14.4}{\rmdefault}{\mddefault}{\updefault}{\color[rgb]{0,0,0}$e$}%
}}}}
\put(1951,-11161){\makebox(0,0)[lb]{\smash{{\SetFigFont{12}{14.4}{\rmdefault}{\mddefault}{\updefault}{\color[rgb]{0,0,0}$Z_4$}%
}}}}
\put(5251,-11161){\makebox(0,0)[lb]{\smash{{\SetFigFont{12}{14.4}{\rmdefault}{\mddefault}{\updefault}{\color[rgb]{0,0,0}$Z_5$}%
}}}}
\put(8626,-11161){\makebox(0,0)[lb]{\smash{{\SetFigFont{12}{14.4}{\rmdefault}{\mddefault}{\updefault}{\color[rgb]{0,0,0}$Z_6$}%
}}}}
\put(10201,-8461){\makebox(0,0)[lb]{\smash{{\SetFigFont{12}{14.4}{\rmdefault}{\mddefault}{\updefault}{\color[rgb]{0,0,0}$b$}%
}}}}
\put(10276,-10336){\makebox(0,0)[lb]{\smash{{\SetFigFont{12}{14.4}{\rmdefault}{\mddefault}{\updefault}{\color[rgb]{0,0,0}$a$}%
}}}}
\put(9676,-9286){\makebox(0,0)[lb]{\smash{{\SetFigFont{12}{14.4}{\rmdefault}{\mddefault}{\updefault}{\color[rgb]{0,0,0}$g$}%
}}}}
\put(9526,-11086){\makebox(0,0)[lb]{\smash{{\SetFigFont{12}{14.4}{\rmdefault}{\mddefault}{\updefault}{\color[rgb]{0,0,0}$h$}%
}}}}
\put(8326,-8386){\makebox(0,0)[lb]{\smash{{\SetFigFont{12}{14.4}{\rmdefault}{\mddefault}{\updefault}{\color[rgb]{0,0,0}$e$}%
}}}}
\put(8101,-10261){\makebox(0,0)[lb]{\smash{{\SetFigFont{12}{14.4}{\rmdefault}{\mddefault}{\updefault}{\color[rgb]{0,0,0}$c$}%
}}}}
\put(7651,-11086){\makebox(0,0)[lb]{\smash{{\SetFigFont{12}{14.4}{\rmdefault}{\mddefault}{\updefault}{\color[rgb]{0,0,0}$f$}%
}}}}
\put(2326,-6211){\makebox(0,0)[lb]{\smash{{\SetFigFont{12}{14.4}{\rmdefault}{\mddefault}{\updefault}{\color[rgb]{0,0,0}$1$}%
}}}}
\put(5551,-6661){\makebox(0,0)[lb]{\smash{{\SetFigFont{12}{14.4}{\rmdefault}{\mddefault}{\updefault}{\color[rgb]{0,0,0}$2$}%
}}}}
\put(5626,-9511){\makebox(0,0)[lb]{\smash{{\SetFigFont{12}{14.4}{\rmdefault}{\mddefault}{\updefault}{\color[rgb]{0,0,0}$5$}%
}}}}
\put(8926,-9511){\makebox(0,0)[lb]{\smash{{\SetFigFont{12}{14.4}{\rmdefault}{\mddefault}{\updefault}{\color[rgb]{0,0,0}$6$}%
}}}}
\put(6376,-10111){\makebox(0,0)[lb]{\smash{{\SetFigFont{12}{14.4}{\rmdefault}{\mddefault}{\updefault}{\color[rgb]{0,0,0}$p$}%
}}}}
\put(6976,-9511){\makebox(0,0)[lb]{\smash{{\SetFigFont{12}{14.4}{\rmdefault}{\mddefault}{\updefault}{\color[rgb]{0,0,0}$q$}%
}}}}
\put(5176,-9286){\makebox(0,0)[lb]{\smash{{\SetFigFont{12}{14.4}{\rmdefault}{\mddefault}{\updefault}{\color[rgb]{0,0,0}$s$}%
}}}}
\put(4276,-10111){\makebox(0,0)[lb]{\smash{{\SetFigFont{12}{14.4}{\rmdefault}{\mddefault}{\updefault}{\color[rgb]{0,0,0}$r$}%
}}}}
\put(10276,-9436){\makebox(0,0)[lb]{\smash{{\SetFigFont{12}{14.4}{\rmdefault}{\mddefault}{\updefault}{\color[rgb]{0,0,0}$q$}%
}}}}
\put(8476,-9286){\makebox(0,0)[lb]{\smash{{\SetFigFont{12}{14.4}{\rmdefault}{\mddefault}{\updefault}{\color[rgb]{0,0,0}$r$}%
}}}}
\put(7576,-10036){\makebox(0,0)[lb]{\smash{{\SetFigFont{12}{14.4}{\rmdefault}{\mddefault}{\updefault}{\color[rgb]{0,0,0}$p$}%
}}}}
\put(9676,-10036){\makebox(0,0)[lb]{\smash{{\SetFigFont{12}{14.4}{\rmdefault}{\mddefault}{\updefault}{\color[rgb]{0,0,0}$s$}%
}}}}
\put(4876,-10111){\makebox(0,0)[lb]{\smash{{\SetFigFont{12}{14.4}{\rmdefault}{\mddefault}{\updefault}{\color[rgb]{0,0,0}$h$}%
}}}}
\end{picture}%

%% file: egch208.pstex_t
\begin{picture}(0,0)%
\includegraphics{egch208.pstex}%
\end{picture}%
\setlength{\unitlength}{3947sp}%
\begingroup\makeatletter\ifx\SetFigFont\undefined%
\gdef\SetFigFont#1#2#3#4#5{%
  \reset@font\fontsize{#1}{#2pt}%
  \fontfamily{#3}\fontseries{#4}\fontshape{#5}%
  \selectfont}%
\fi\endgroup%
\begin{picture}(9405,6540)(886,-11455)
\put(1951,-7936){\makebox(0,0)[lb]{\smash{{\SetFigFont{12}{14.4}{\rmdefault}{\mddefault}{\updefault}{\color[rgb]{0,0,0}$Z_1$}%
}}}}
\put(5176,-7936){\makebox(0,0)[lb]{\smash{{\SetFigFont{12}{14.4}{\rmdefault}{\mddefault}{\updefault}{\color[rgb]{0,0,0}$Z_2$}%
}}}}
\put(8551,-8011){\makebox(0,0)[lb]{\smash{{\SetFigFont{12}{14.4}{\rmdefault}{\mddefault}{\updefault}{\color[rgb]{0,0,0}$Z_3$}%
}}}}
\put(1951,-11311){\makebox(0,0)[lb]{\smash{{\SetFigFont{12}{14.4}{\rmdefault}{\mddefault}{\updefault}{\color[rgb]{0,0,0}$Z_4$}%
}}}}
\put(5176,-11386){\makebox(0,0)[lb]{\smash{{\SetFigFont{12}{14.4}{\rmdefault}{\mddefault}{\updefault}{\color[rgb]{0,0,0}$Z_5$}%
}}}}
\put(8626,-11311){\makebox(0,0)[lb]{\smash{{\SetFigFont{12}{14.4}{\rmdefault}{\mddefault}{\updefault}{\color[rgb]{0,0,0}$Z_6$}%
}}}}
\put(2776,-5986){\makebox(0,0)[lb]{\smash{{\SetFigFont{12}{14.4}{\rmdefault}{\mddefault}{\updefault}{\color[rgb]{0,0,0}$d$}%
}}}}
\put(4276,-5761){\makebox(0,0)[lb]{\smash{{\SetFigFont{12}{14.4}{\rmdefault}{\mddefault}{\updefault}{\color[rgb]{0,0,0}$d$}%
}}}}
\put(7576,-5836){\makebox(0,0)[lb]{\smash{{\SetFigFont{12}{14.4}{\rmdefault}{\mddefault}{\updefault}{\color[rgb]{0,0,0}$d$}%
}}}}
\put(7576,-9211){\makebox(0,0)[lb]{\smash{{\SetFigFont{12}{14.4}{\rmdefault}{\mddefault}{\updefault}{\color[rgb]{0,0,0}$d$}%
}}}}
\put(4276,-9136){\makebox(0,0)[lb]{\smash{{\SetFigFont{12}{14.4}{\rmdefault}{\mddefault}{\updefault}{\color[rgb]{0,0,0}$d$}%
}}}}
\put(3076,-9211){\makebox(0,0)[lb]{\smash{{\SetFigFont{12}{14.4}{\rmdefault}{\mddefault}{\updefault}{\color[rgb]{0,0,0}$d$}%
}}}}
\put(3676,-5161){\makebox(0,0)[lb]{\smash{{\SetFigFont{12}{14.4}{\rmdefault}{\mddefault}{\updefault}{\color[rgb]{0,0,0}$a$}%
}}}}
\put(5026,-5086){\makebox(0,0)[lb]{\smash{{\SetFigFont{12}{14.4}{\rmdefault}{\mddefault}{\updefault}{\color[rgb]{0,0,0}$a$}%
}}}}
\put(10201,-5161){\makebox(0,0)[lb]{\smash{{\SetFigFont{12}{14.4}{\rmdefault}{\mddefault}{\updefault}{\color[rgb]{0,0,0}$a$}%
}}}}
\put(10201,-8461){\makebox(0,0)[lb]{\smash{{\SetFigFont{12}{14.4}{\rmdefault}{\mddefault}{\updefault}{\color[rgb]{0,0,0}$a$}%
}}}}
\put(5101,-8386){\makebox(0,0)[lb]{\smash{{\SetFigFont{12}{14.4}{\rmdefault}{\mddefault}{\updefault}{\color[rgb]{0,0,0}$a$}%
}}}}
\put(3601,-8461){\makebox(0,0)[lb]{\smash{{\SetFigFont{12}{14.4}{\rmdefault}{\mddefault}{\updefault}{\color[rgb]{0,0,0}$a$}%
}}}}
\put(1801,-5086){\makebox(0,0)[lb]{\smash{{\SetFigFont{12}{14.4}{\rmdefault}{\mddefault}{\updefault}{\color[rgb]{0,0,0}$c$}%
}}}}
\put(6376,-5911){\makebox(0,0)[lb]{\smash{{\SetFigFont{12}{14.4}{\rmdefault}{\mddefault}{\updefault}{\color[rgb]{0,0,0}$c$}%
}}}}
\put(9676,-5911){\makebox(0,0)[lb]{\smash{{\SetFigFont{12}{14.4}{\rmdefault}{\mddefault}{\updefault}{\color[rgb]{0,0,0}$c$}%
}}}}
\put(9676,-9286){\makebox(0,0)[lb]{\smash{{\SetFigFont{12}{14.4}{\rmdefault}{\mddefault}{\updefault}{\color[rgb]{0,0,0}$c$}%
}}}}
\put(6376,-9211){\makebox(0,0)[lb]{\smash{{\SetFigFont{12}{14.4}{\rmdefault}{\mddefault}{\updefault}{\color[rgb]{0,0,0}$c$}%
}}}}
\put(1726,-8386){\makebox(0,0)[lb]{\smash{{\SetFigFont{12}{14.4}{\rmdefault}{\mddefault}{\updefault}{\color[rgb]{0,0,0}$c$}%
}}}}
\put(901,-5686){\makebox(0,0)[lb]{\smash{{\SetFigFont{12}{14.4}{\rmdefault}{\mddefault}{\updefault}{\color[rgb]{0,0,0}$g$}%
}}}}
\put(6901,-5161){\makebox(0,0)[lb]{\smash{{\SetFigFont{12}{14.4}{\rmdefault}{\mddefault}{\updefault}{\color[rgb]{0,0,0}$g$}%
}}}}
\put(8326,-5086){\makebox(0,0)[lb]{\smash{{\SetFigFont{12}{14.4}{\rmdefault}{\mddefault}{\updefault}{\color[rgb]{0,0,0}$g$}%
}}}}
\put(8326,-8386){\makebox(0,0)[lb]{\smash{{\SetFigFont{12}{14.4}{\rmdefault}{\mddefault}{\updefault}{\color[rgb]{0,0,0}$g$}%
}}}}
\put(6901,-8461){\makebox(0,0)[lb]{\smash{{\SetFigFont{12}{14.4}{\rmdefault}{\mddefault}{\updefault}{\color[rgb]{0,0,0}$g$}%
}}}}
\put(976,-8986){\makebox(0,0)[lb]{\smash{{\SetFigFont{12}{14.4}{\rmdefault}{\mddefault}{\updefault}{\color[rgb]{0,0,0}$g$}%
}}}}
\put(1876,-10186){\makebox(0,0)[lb]{\smash{{\SetFigFont{12}{14.4}{\rmdefault}{\mddefault}{\updefault}{\color[rgb]{0,0,0}$f$}%
}}}}
\put(6226,-11086){\makebox(0,0)[lb]{\smash{{\SetFigFont{12}{14.4}{\rmdefault}{\mddefault}{\updefault}{\color[rgb]{0,0,0}$f$}%
}}}}
\put(9526,-11086){\makebox(0,0)[lb]{\smash{{\SetFigFont{12}{14.4}{\rmdefault}{\mddefault}{\updefault}{\color[rgb]{0,0,0}$f$}%
}}}}
\put(7726,-7786){\makebox(0,0)[lb]{\smash{{\SetFigFont{12}{14.4}{\rmdefault}{\mddefault}{\updefault}{\color[rgb]{0,0,0}$f$}%
}}}}
\put(4426,-7786){\makebox(0,0)[lb]{\smash{{\SetFigFont{12}{14.4}{\rmdefault}{\mddefault}{\updefault}{\color[rgb]{0,0,0}$f$}%
}}}}
\put(2851,-7786){\makebox(0,0)[lb]{\smash{{\SetFigFont{12}{14.4}{\rmdefault}{\mddefault}{\updefault}{\color[rgb]{0,0,0}$f$}%
}}}}
\put(3676,-7036){\makebox(0,0)[lb]{\smash{{\SetFigFont{12}{14.4}{\rmdefault}{\mddefault}{\updefault}{\color[rgb]{0,0,0}$b$}%
}}}}
\put(4801,-6961){\makebox(0,0)[lb]{\smash{{\SetFigFont{12}{14.4}{\rmdefault}{\mddefault}{\updefault}{\color[rgb]{0,0,0}$b$}%
}}}}
\put(10276,-7036){\makebox(0,0)[lb]{\smash{{\SetFigFont{12}{14.4}{\rmdefault}{\mddefault}{\updefault}{\color[rgb]{0,0,0}$b$}%
}}}}
\put(8101,-10261){\makebox(0,0)[lb]{\smash{{\SetFigFont{12}{14.4}{\rmdefault}{\mddefault}{\updefault}{\color[rgb]{0,0,0}$b$}%
}}}}
\put(6976,-10336){\makebox(0,0)[lb]{\smash{{\SetFigFont{12}{14.4}{\rmdefault}{\mddefault}{\updefault}{\color[rgb]{0,0,0}$b$}%
}}}}
\put(1126,-11086){\makebox(0,0)[lb]{\smash{{\SetFigFont{12}{14.4}{\rmdefault}{\mddefault}{\updefault}{\color[rgb]{0,0,0}$b$}%
}}}}
\put(2926,-11086){\makebox(0,0)[lb]{\smash{{\SetFigFont{12}{14.4}{\rmdefault}{\mddefault}{\updefault}{\color[rgb]{0,0,0}$e$}%
}}}}
\put(4426,-11086){\makebox(0,0)[lb]{\smash{{\SetFigFont{12}{14.4}{\rmdefault}{\mddefault}{\updefault}{\color[rgb]{0,0,0}$e$}%
}}}}
\put(7651,-11086){\makebox(0,0)[lb]{\smash{{\SetFigFont{12}{14.4}{\rmdefault}{\mddefault}{\updefault}{\color[rgb]{0,0,0}$e$}%
}}}}
\put(9526,-7786){\makebox(0,0)[lb]{\smash{{\SetFigFont{12}{14.4}{\rmdefault}{\mddefault}{\updefault}{\color[rgb]{0,0,0}$e$}%
}}}}
\put(6226,-7786){\makebox(0,0)[lb]{\smash{{\SetFigFont{12}{14.4}{\rmdefault}{\mddefault}{\updefault}{\color[rgb]{0,0,0}$e$}%
}}}}
\put(1501,-6961){\makebox(0,0)[lb]{\smash{{\SetFigFont{12}{14.4}{\rmdefault}{\mddefault}{\updefault}{\color[rgb]{0,0,0}$e$}%
}}}}
\put(1201,-7786){\makebox(0,0)[lb]{\smash{{\SetFigFont{12}{14.4}{\rmdefault}{\mddefault}{\updefault}{\color[rgb]{0,0,0}$h$}%
}}}}
\put(6976,-7111){\makebox(0,0)[lb]{\smash{{\SetFigFont{12}{14.4}{\rmdefault}{\mddefault}{\updefault}{\color[rgb]{0,0,0}$h$}%
}}}}
\put(8101,-7036){\makebox(0,0)[lb]{\smash{{\SetFigFont{12}{14.4}{\rmdefault}{\mddefault}{\updefault}{\color[rgb]{0,0,0}$h$}%
}}}}
\put(3676,-10336){\makebox(0,0)[lb]{\smash{{\SetFigFont{12}{14.4}{\rmdefault}{\mddefault}{\updefault}{\color[rgb]{0,0,0}$h$}%
}}}}
\put(4876,-10261){\makebox(0,0)[lb]{\smash{{\SetFigFont{12}{14.4}{\rmdefault}{\mddefault}{\updefault}{\color[rgb]{0,0,0}$h$}%
}}}}
\put(10276,-10336){\makebox(0,0)[lb]{\smash{{\SetFigFont{12}{14.4}{\rmdefault}{\mddefault}{\updefault}{\color[rgb]{0,0,0}$h$}%
}}}}
\put(2326,-5386){\makebox(0,0)[lb]{\smash{{\SetFigFont{12}{14.4}{\rmdefault}{\mddefault}{\updefault}{\color[rgb]{0,0,0}$q$}%
}}}}
\put(2251,-7486){\makebox(0,0)[lb]{\smash{{\SetFigFont{12}{14.4}{\rmdefault}{\mddefault}{\updefault}{\color[rgb]{0,0,0}$r$}%
}}}}
\put(2326,-8686){\makebox(0,0)[lb]{\smash{{\SetFigFont{12}{14.4}{\rmdefault}{\mddefault}{\updefault}{\color[rgb]{0,0,0}$q$}%
}}}}
\put(2251,-10786){\makebox(0,0)[lb]{\smash{{\SetFigFont{12}{14.4}{\rmdefault}{\mddefault}{\updefault}{\color[rgb]{0,0,0}$r$}%
}}}}
\put(5626,-5386){\makebox(0,0)[lb]{\smash{{\SetFigFont{12}{14.4}{\rmdefault}{\mddefault}{\updefault}{\color[rgb]{0,0,0}$s$}%
}}}}
\put(5551,-7486){\makebox(0,0)[lb]{\smash{{\SetFigFont{12}{14.4}{\rmdefault}{\mddefault}{\updefault}{\color[rgb]{0,0,0}$t$}%
}}}}
\put(5626,-6886){\makebox(0,0)[lb]{\smash{{\SetFigFont{12}{14.4}{\rmdefault}{\mddefault}{\updefault}{\color[rgb]{0,0,0}$w$}%
}}}}
\put(5551,-6061){\makebox(0,0)[lb]{\smash{{\SetFigFont{12}{14.4}{\rmdefault}{\mddefault}{\updefault}{\color[rgb]{0,0,0}$x$}%
}}}}
\put(8926,-5386){\makebox(0,0)[lb]{\smash{{\SetFigFont{12}{14.4}{\rmdefault}{\mddefault}{\updefault}{\color[rgb]{0,0,0}$u$}%
}}}}
\put(8851,-7486){\makebox(0,0)[lb]{\smash{{\SetFigFont{12}{14.4}{\rmdefault}{\mddefault}{\updefault}{\color[rgb]{0,0,0}$v$}%
}}}}
\put(8926,-6886){\makebox(0,0)[lb]{\smash{{\SetFigFont{12}{14.4}{\rmdefault}{\mddefault}{\updefault}{\color[rgb]{0,0,0}$w$}%
}}}}
\put(9301,-5986){\makebox(0,0)[lb]{\smash{{\SetFigFont{12}{14.4}{\rmdefault}{\mddefault}{\updefault}{\color[rgb]{0,0,0}$x$}%
}}}}
\put(2101,-9736){\makebox(0,0)[lb]{\smash{{\SetFigFont{12}{14.4}{\rmdefault}{\mddefault}{\updefault}{\color[rgb]{0,0,0}$4$}%
}}}}
\put(5626,-8686){\makebox(0,0)[lb]{\smash{{\SetFigFont{12}{14.4}{\rmdefault}{\mddefault}{\updefault}{\color[rgb]{0,0,0}$s$}%
}}}}
\put(5551,-10786){\makebox(0,0)[lb]{\smash{{\SetFigFont{12}{14.4}{\rmdefault}{\mddefault}{\updefault}{\color[rgb]{0,0,0}$t$}%
}}}}
\put(5626,-10186){\makebox(0,0)[lb]{\smash{{\SetFigFont{12}{14.4}{\rmdefault}{\mddefault}{\updefault}{\color[rgb]{0,0,0}$y$}%
}}}}
\put(8926,-8686){\makebox(0,0)[lb]{\smash{{\SetFigFont{12}{14.4}{\rmdefault}{\mddefault}{\updefault}{\color[rgb]{0,0,0}$u$}%
}}}}
\put(8851,-10786){\makebox(0,0)[lb]{\smash{{\SetFigFont{12}{14.4}{\rmdefault}{\mddefault}{\updefault}{\color[rgb]{0,0,0}$v$}%
}}}}
\put(8926,-10186){\makebox(0,0)[lb]{\smash{{\SetFigFont{12}{14.4}{\rmdefault}{\mddefault}{\updefault}{\color[rgb]{0,0,0}$y$}%
}}}}
\put(5851,-9436){\makebox(0,0)[lb]{\smash{{\SetFigFont{12}{14.4}{\rmdefault}{\mddefault}{\updefault}{\color[rgb]{0,0,0}$z$}%
}}}}
\put(9226,-9286){\makebox(0,0)[lb]{\smash{{\SetFigFont{12}{14.4}{\rmdefault}{\mddefault}{\updefault}{\color[rgb]{0,0,0}$z$}%
}}}}
\end{picture}%

%% file: egch201.pstex_t
\begin{picture}(0,0)%
\includegraphics{egch201.pstex}%
\end{picture}%
\setlength{\unitlength}{3947sp}%
\begingroup\makeatletter\ifx\SetFigFont\undefined%
\gdef\SetFigFont#1#2#3#4#5{%
  \reset@font\fontsize{#1}{#2pt}%
  \fontfamily{#3}\fontseries{#4}\fontshape{#5}%
  \selectfont}%
\fi\endgroup%
\begin{picture}(9405,6390)(886,-11305)
\put(1951,-7936){\makebox(0,0)[lb]{\smash{{\SetFigFont{12}{14.4}{\rmdefault}{\mddefault}{\updefault}{\color[rgb]{0,0,0}$Z_1$}%
}}}}
\put(5176,-7936){\makebox(0,0)[lb]{\smash{{\SetFigFont{12}{14.4}{\rmdefault}{\mddefault}{\updefault}{\color[rgb]{0,0,0}$Z_2$}%
}}}}
\put(8551,-8011){\makebox(0,0)[lb]{\smash{{\SetFigFont{12}{14.4}{\rmdefault}{\mddefault}{\updefault}{\color[rgb]{0,0,0}$Z_3$}%
}}}}
\put(7576,-5836){\makebox(0,0)[lb]{\smash{{\SetFigFont{12}{14.4}{\rmdefault}{\mddefault}{\updefault}{\color[rgb]{0,0,0}$d$}%
}}}}
\put(7576,-9211){\makebox(0,0)[lb]{\smash{{\SetFigFont{12}{14.4}{\rmdefault}{\mddefault}{\updefault}{\color[rgb]{0,0,0}$d$}%
}}}}
\put(3076,-9211){\makebox(0,0)[lb]{\smash{{\SetFigFont{12}{14.4}{\rmdefault}{\mddefault}{\updefault}{\color[rgb]{0,0,0}$d$}%
}}}}
\put(3676,-5161){\makebox(0,0)[lb]{\smash{{\SetFigFont{12}{14.4}{\rmdefault}{\mddefault}{\updefault}{\color[rgb]{0,0,0}$a$}%
}}}}
\put(10201,-5161){\makebox(0,0)[lb]{\smash{{\SetFigFont{12}{14.4}{\rmdefault}{\mddefault}{\updefault}{\color[rgb]{0,0,0}$a$}%
}}}}
\put(10201,-8461){\makebox(0,0)[lb]{\smash{{\SetFigFont{12}{14.4}{\rmdefault}{\mddefault}{\updefault}{\color[rgb]{0,0,0}$a$}%
}}}}
\put(3601,-8461){\makebox(0,0)[lb]{\smash{{\SetFigFont{12}{14.4}{\rmdefault}{\mddefault}{\updefault}{\color[rgb]{0,0,0}$a$}%
}}}}
\put(1726,-8386){\makebox(0,0)[lb]{\smash{{\SetFigFont{12}{14.4}{\rmdefault}{\mddefault}{\updefault}{\color[rgb]{0,0,0}$c$}%
}}}}
\put(901,-5686){\makebox(0,0)[lb]{\smash{{\SetFigFont{12}{14.4}{\rmdefault}{\mddefault}{\updefault}{\color[rgb]{0,0,0}$g$}%
}}}}
\put(976,-8986){\makebox(0,0)[lb]{\smash{{\SetFigFont{12}{14.4}{\rmdefault}{\mddefault}{\updefault}{\color[rgb]{0,0,0}$g$}%
}}}}
\put(6226,-11086){\makebox(0,0)[lb]{\smash{{\SetFigFont{12}{14.4}{\rmdefault}{\mddefault}{\updefault}{\color[rgb]{0,0,0}$f$}%
}}}}
\put(7726,-7786){\makebox(0,0)[lb]{\smash{{\SetFigFont{12}{14.4}{\rmdefault}{\mddefault}{\updefault}{\color[rgb]{0,0,0}$f$}%
}}}}
\put(2851,-7786){\makebox(0,0)[lb]{\smash{{\SetFigFont{12}{14.4}{\rmdefault}{\mddefault}{\updefault}{\color[rgb]{0,0,0}$f$}%
}}}}
\put(3676,-7036){\makebox(0,0)[lb]{\smash{{\SetFigFont{12}{14.4}{\rmdefault}{\mddefault}{\updefault}{\color[rgb]{0,0,0}$b$}%
}}}}
\put(10276,-7036){\makebox(0,0)[lb]{\smash{{\SetFigFont{12}{14.4}{\rmdefault}{\mddefault}{\updefault}{\color[rgb]{0,0,0}$b$}%
}}}}
\put(6976,-10336){\makebox(0,0)[lb]{\smash{{\SetFigFont{12}{14.4}{\rmdefault}{\mddefault}{\updefault}{\color[rgb]{0,0,0}$b$}%
}}}}
\put(4426,-11086){\makebox(0,0)[lb]{\smash{{\SetFigFont{12}{14.4}{\rmdefault}{\mddefault}{\updefault}{\color[rgb]{0,0,0}$e$}%
}}}}
\put(1201,-7786){\makebox(0,0)[lb]{\smash{{\SetFigFont{12}{14.4}{\rmdefault}{\mddefault}{\updefault}{\color[rgb]{0,0,0}$h$}%
}}}}
\put(1801,-5086){\makebox(0,0)[lb]{\smash{{\SetFigFont{12}{14.4}{\rmdefault}{\mddefault}{\updefault}{\color[rgb]{0,0,0}$c$}%
}}}}
\put(3076,-5911){\makebox(0,0)[lb]{\smash{{\SetFigFont{12}{14.4}{\rmdefault}{\mddefault}{\updefault}{\color[rgb]{0,0,0}$d$}%
}}}}
\put(2326,-5386){\makebox(0,0)[lb]{\smash{{\SetFigFont{12}{14.4}{\rmdefault}{\mddefault}{\updefault}{\color[rgb]{0,0,0}$u$}%
}}}}
\put(2251,-7486){\makebox(0,0)[lb]{\smash{{\SetFigFont{12}{14.4}{\rmdefault}{\mddefault}{\updefault}{\color[rgb]{0,0,0}$v$}%
}}}}
\put(2326,-8686){\makebox(0,0)[lb]{\smash{{\SetFigFont{12}{14.4}{\rmdefault}{\mddefault}{\updefault}{\color[rgb]{0,0,0}$u$}%
}}}}
\put(2251,-10786){\makebox(0,0)[lb]{\smash{{\SetFigFont{12}{14.4}{\rmdefault}{\mddefault}{\updefault}{\color[rgb]{0,0,0}$v$}%
}}}}
\put(3001,-10111){\makebox(0,0)[lb]{\smash{{\SetFigFont{12}{14.4}{\rmdefault}{\mddefault}{\updefault}{\color[rgb]{0,0,0}$q$}%
}}}}
\put(1801,-9586){\makebox(0,0)[lb]{\smash{{\SetFigFont{12}{14.4}{\rmdefault}{\mddefault}{\updefault}{\color[rgb]{0,0,0}$s$}%
}}}}
\put(976,-10111){\makebox(0,0)[lb]{\smash{{\SetFigFont{12}{14.4}{\rmdefault}{\mddefault}{\updefault}{\color[rgb]{0,0,0}$t$}%
}}}}
\put(6376,-10036){\makebox(0,0)[lb]{\smash{{\SetFigFont{12}{14.4}{\rmdefault}{\mddefault}{\updefault}{\color[rgb]{0,0,0}$q$}%
}}}}
\put(6976,-9511){\makebox(0,0)[lb]{\smash{{\SetFigFont{12}{14.4}{\rmdefault}{\mddefault}{\updefault}{\color[rgb]{0,0,0}$r$}%
}}}}
\put(5176,-9511){\makebox(0,0)[lb]{\smash{{\SetFigFont{12}{14.4}{\rmdefault}{\mddefault}{\updefault}{\color[rgb]{0,0,0}$t$}%
}}}}
\put(4276,-10111){\makebox(0,0)[lb]{\smash{{\SetFigFont{12}{14.4}{\rmdefault}{\mddefault}{\updefault}{\color[rgb]{0,0,0}$s$}%
}}}}
\put(7576,-10036){\makebox(0,0)[lb]{\smash{{\SetFigFont{12}{14.4}{\rmdefault}{\mddefault}{\updefault}{\color[rgb]{0,0,0}$q$}%
}}}}
\put(10276,-9436){\makebox(0,0)[lb]{\smash{{\SetFigFont{12}{14.4}{\rmdefault}{\mddefault}{\updefault}{\color[rgb]{0,0,0}$r$}%
}}}}
\put(8476,-9436){\makebox(0,0)[lb]{\smash{{\SetFigFont{12}{14.4}{\rmdefault}{\mddefault}{\updefault}{\color[rgb]{0,0,0}$s$}%
}}}}
\put(9601,-10111){\makebox(0,0)[lb]{\smash{{\SetFigFont{12}{14.4}{\rmdefault}{\mddefault}{\updefault}{\color[rgb]{0,0,0}$t$}%
}}}}
\put(5626,-5386){\makebox(0,0)[lb]{\smash{{\SetFigFont{12}{14.4}{\rmdefault}{\mddefault}{\updefault}{\color[rgb]{0,0,0}$w$}%
}}}}
\put(5551,-7486){\makebox(0,0)[lb]{\smash{{\SetFigFont{12}{14.4}{\rmdefault}{\mddefault}{\updefault}{\color[rgb]{0,0,0}$x$}%
}}}}
\put(5626,-8686){\makebox(0,0)[lb]{\smash{{\SetFigFont{12}{14.4}{\rmdefault}{\mddefault}{\updefault}{\color[rgb]{0,0,0}$w$}%
}}}}
\put(5551,-10786){\makebox(0,0)[lb]{\smash{{\SetFigFont{12}{14.4}{\rmdefault}{\mddefault}{\updefault}{\color[rgb]{0,0,0}$x$}%
}}}}
\put(8926,-5386){\makebox(0,0)[lb]{\smash{{\SetFigFont{12}{14.4}{\rmdefault}{\mddefault}{\updefault}{\color[rgb]{0,0,0}$y$}%
}}}}
\put(8851,-7486){\makebox(0,0)[lb]{\smash{{\SetFigFont{12}{14.4}{\rmdefault}{\mddefault}{\updefault}{\color[rgb]{0,0,0}$z$}%
}}}}
\put(8851,-10786){\makebox(0,0)[lb]{\smash{{\SetFigFont{12}{14.4}{\rmdefault}{\mddefault}{\updefault}{\color[rgb]{0,0,0}$z$}%
}}}}
\put(6901,-5161){\makebox(0,0)[lb]{\smash{{\SetFigFont{12}{14.4}{\rmdefault}{\mddefault}{\updefault}{\color[rgb]{0,0,0}$a$}%
}}}}
\put(6976,-7111){\makebox(0,0)[lb]{\smash{{\SetFigFont{12}{14.4}{\rmdefault}{\mddefault}{\updefault}{\color[rgb]{0,0,0}$b$}%
}}}}
\put(6376,-5911){\makebox(0,0)[lb]{\smash{{\SetFigFont{12}{14.4}{\rmdefault}{\mddefault}{\updefault}{\color[rgb]{0,0,0}$d$}%
}}}}
\put(6226,-7786){\makebox(0,0)[lb]{\smash{{\SetFigFont{12}{14.4}{\rmdefault}{\mddefault}{\updefault}{\color[rgb]{0,0,0}$f$}%
}}}}
\put(5026,-5086){\makebox(0,0)[lb]{\smash{{\SetFigFont{12}{14.4}{\rmdefault}{\mddefault}{\updefault}{\color[rgb]{0,0,0}$g$}%
}}}}
\put(4276,-5761){\makebox(0,0)[lb]{\smash{{\SetFigFont{12}{14.4}{\rmdefault}{\mddefault}{\updefault}{\color[rgb]{0,0,0}$c$}%
}}}}
\put(4426,-7786){\makebox(0,0)[lb]{\smash{{\SetFigFont{12}{14.4}{\rmdefault}{\mddefault}{\updefault}{\color[rgb]{0,0,0}$e$}%
}}}}
\put(8326,-5086){\makebox(0,0)[lb]{\smash{{\SetFigFont{12}{14.4}{\rmdefault}{\mddefault}{\updefault}{\color[rgb]{0,0,0}$c$}%
}}}}
\put(3676,-10336){\makebox(0,0)[lb]{\smash{{\SetFigFont{12}{14.4}{\rmdefault}{\mddefault}{\updefault}{\color[rgb]{0,0,0}$b$}%
}}}}
\put(2926,-11086){\makebox(0,0)[lb]{\smash{{\SetFigFont{12}{14.4}{\rmdefault}{\mddefault}{\updefault}{\color[rgb]{0,0,0}$f$}%
}}}}
\put(1876,-10186){\makebox(0,0)[lb]{\smash{{\SetFigFont{12}{14.4}{\rmdefault}{\mddefault}{\updefault}{\color[rgb]{0,0,0}$e$}%
}}}}
\put(1126,-11086){\makebox(0,0)[lb]{\smash{{\SetFigFont{12}{14.4}{\rmdefault}{\mddefault}{\updefault}{\color[rgb]{0,0,0}$h$}%
}}}}
\put(6901,-8461){\makebox(0,0)[lb]{\smash{{\SetFigFont{12}{14.4}{\rmdefault}{\mddefault}{\updefault}{\color[rgb]{0,0,0}$a$}%
}}}}
\put(6376,-9211){\makebox(0,0)[lb]{\smash{{\SetFigFont{12}{14.4}{\rmdefault}{\mddefault}{\updefault}{\color[rgb]{0,0,0}$d$}%
}}}}
\put(5101,-8386){\makebox(0,0)[lb]{\smash{{\SetFigFont{12}{14.4}{\rmdefault}{\mddefault}{\updefault}{\color[rgb]{0,0,0}$g$}%
}}}}
\put(4876,-10261){\makebox(0,0)[lb]{\smash{{\SetFigFont{12}{14.4}{\rmdefault}{\mddefault}{\updefault}{\color[rgb]{0,0,0}$h$}%
}}}}
\put(4276,-9136){\makebox(0,0)[lb]{\smash{{\SetFigFont{12}{14.4}{\rmdefault}{\mddefault}{\updefault}{\color[rgb]{0,0,0}$c$}%
}}}}
\put(10276,-10336){\makebox(0,0)[lb]{\smash{{\SetFigFont{12}{14.4}{\rmdefault}{\mddefault}{\updefault}{\color[rgb]{0,0,0}$b$}%
}}}}
\put(9676,-5911){\makebox(0,0)[lb]{\smash{{\SetFigFont{12}{14.4}{\rmdefault}{\mddefault}{\updefault}{\color[rgb]{0,0,0}$g$}%
}}}}
\put(9526,-7786){\makebox(0,0)[lb]{\smash{{\SetFigFont{12}{14.4}{\rmdefault}{\mddefault}{\updefault}{\color[rgb]{0,0,0}$h$}%
}}}}
\put(9601,-9211){\makebox(0,0)[lb]{\smash{{\SetFigFont{12}{14.4}{\rmdefault}{\mddefault}{\updefault}{\color[rgb]{0,0,0}$g$}%
}}}}
\put(9526,-11086){\makebox(0,0)[lb]{\smash{{\SetFigFont{12}{14.4}{\rmdefault}{\mddefault}{\updefault}{\color[rgb]{0,0,0}$h$}%
}}}}
\put(8326,-8386){\makebox(0,0)[lb]{\smash{{\SetFigFont{12}{14.4}{\rmdefault}{\mddefault}{\updefault}{\color[rgb]{0,0,0}$c$}%
}}}}
\put(8101,-10261){\makebox(0,0)[lb]{\smash{{\SetFigFont{12}{14.4}{\rmdefault}{\mddefault}{\updefault}{\color[rgb]{0,0,0}$e$}%
}}}}
\put(7651,-11086){\makebox(0,0)[lb]{\smash{{\SetFigFont{12}{14.4}{\rmdefault}{\mddefault}{\updefault}{\color[rgb]{0,0,0}$f$}%
}}}}
\put(3676,-9436){\makebox(0,0)[lb]{\smash{{\SetFigFont{12}{14.4}{\rmdefault}{\mddefault}{\updefault}{\color[rgb]{0,0,0}$r$}%
}}}}
\put(1576,-7036){\makebox(0,0)[lb]{\smash{{\SetFigFont{12}{14.4}{\rmdefault}{\mddefault}{\updefault}{\color[rgb]{0,0,0}$e$}%
}}}}
\put(4876,-7036){\makebox(0,0)[lb]{\smash{{\SetFigFont{12}{14.4}{\rmdefault}{\mddefault}{\updefault}{\color[rgb]{0,0,0}$h$}%
}}}}
\put(8176,-7036){\makebox(0,0)[lb]{\smash{{\SetFigFont{12}{14.4}{\rmdefault}{\mddefault}{\updefault}{\color[rgb]{0,0,0}$e$}%
}}}}
\put(8926,-8686){\makebox(0,0)[lb]{\smash{{\SetFigFont{12}{14.4}{\rmdefault}{\mddefault}{\updefault}{\color[rgb]{0,0,0}$y$}%
}}}}
\put(1951,-11161){\makebox(0,0)[lb]{\smash{{\SetFigFont{12}{14.4}{\rmdefault}{\mddefault}{\updefault}{\color[rgb]{0,0,0}$Z_4$}%
}}}}
\put(5251,-11236){\makebox(0,0)[lb]{\smash{{\SetFigFont{12}{14.4}{\rmdefault}{\mddefault}{\updefault}{\color[rgb]{0,0,0}$Z_5$}%
}}}}
\put(8701,-11161){\makebox(0,0)[lb]{\smash{{\SetFigFont{12}{14.4}{\rmdefault}{\mddefault}{\updefault}{\color[rgb]{0,0,0}$Z_6$}%
}}}}
\end{picture}%